\documentclass{birkmult}
 \newtheorem{thm}{Theorem}[section]
 \newtheorem{cor}[thm]{Corollary}
 \newtheorem{lem}[thm]{Lemma}
 \newtheorem{prop}[thm]{Proposition}
 \theoremstyle{definition}
 \newtheorem{defn}[thm]{Definition}
 \theoremstyle{remark}
 \newtheorem{rem}[thm]{Remark}
 
 \numberwithin{equation}{section}

\newcommand{\cN}{{\mathcal N}}
\newcommand{\cM}{{\mathcal M}}
\newcommand{\cC}{{\mathcal C}}
\newcommand{\cS}{{\mathcal S}}
\newcommand{\cI}{{\mathcal I}}
\newcommand{\gB}{{\mathfrak B}}
\newcommand{\gH}{{\mathfrak H}}
\newcommand{\gF}{{\mathfrak F}}
\newcommand{\gG}{{\mathfrak G}}
\newcommand{\gD}{{\mathfrak D}}
\newcommand{\gL}{{\mathfrak L}}
\newcommand{\gN}{{\mathfrak N}}
\newcommand{\dD}{{\mathbb D}}
\newcommand{\dE}{{\mathbb E}}
\newcommand{\dT}{{\mathbb T}}
\newcommand{\dC}{{\mathbb C}}
\newcommand{\wh}{{\widehat{\hphantom{t}}}}

\begin{document}
%
%
%
%
%
%
%
%
%
\title[Pseudocontinuable Schur Functions]
      {Shift Operators Contained in Contractions, \\
       Pseudocontinuable Schur Functions and \\ 
       Orthogonal Systems on the Unit Circle}

\author[V.K. Dubovoy]{Vladimir K. Dubovoy}
\address{%
Department of Mathematics and Mechanics\\
Kharkov State University\\
Svobody Square 4 \\
UA-61077 Kharkov, 
Ukraine}
\email{DUBOVOY@online.kharkiv.com}

\author[B. Fritzsche, B. Kirstein]{Bernd Fritzsche, Bernd Kirstein\\}
\address{%
\vspace*{-7.3mm} \\
Fakult\"at f\"ur Mathematik und Informatik\\
Universit\"at Leipzig \\
Postfach 10 09 20 \\
D-04009 Leipzig, 
Germany}
\email{$\{$fritzsche,kirstein$\}$@math.uni-leipzig.de}

\author[A. Lasarow]{Andreas Lasarow}
\address{%
 Departement Computerwetenschappen \\
 K.U. Leuven \\
 Celestijnenlaan 200A - postbus: 02402 \\
 B-3001 Heverlee (Leuven), 
 Belgium}
\email{Andreas.Lasarow@cs.kuleuven.be}

\thanks{The work of the fourth author of the present paper was supported by the
German Research Foundation (Deutsche Forschungsgemeinschaft) on badge LA 1386/2--1.}



\date{\today}

\dedicatory{{\Large Dedicated to D.Z. Arov on the occasion of his 75th birthday}}

\begin{abstract}
The main aim of this paper is to establish the connection between well-known criteria for the pseudocontinuability
of a non-inner Schur function $\Theta$ in the unit disk (see Theorems~\ref{t3.9} and \ref{t4.2}). 
In a canonical way we associate a probability measure $\mu$ on the unit circle with $\Theta$.
One of the two criteria will be reformulated in the face of $\mu$, whereas the other one is drafted in view of a 
completely non--unitary contraction $T$ having $\Theta$
as its corresponding charac\-teris\-tic function. Our main result clarifies an immediate connection between 
the above-mentioned two criteria. 
For this reason, we construct a special orthogonal basis in the space
$L_{\mu}^2$ and rewrite these criteria in terms of this orthogonal basis.
(see Theorem~\ref{4.3}). 
\end{abstract}

\maketitle

\vspace{-0.5cm}

\setcounter{section}{-1}
\section{Introduction}
\label{s0}


The central topic of this paper is the discussion of pseudocontinuability for Schur functions in the unit disk. 
Pseudocontinuability is a particular type of meromorphic continuation for functions from the meromorphic
Nevanlinna class in the unit disk. This concept originated in Shapiro's papers \cite{14a}, \cite{SH} and was then
systematically discussed a little bit later in the context of invariant subspaces for the backward shift in the
landmark paper Douglas/Shapiro/Shields \cite{DSS}. (Regarding modern treatments of this area we refer
the reader to the monographs Cima/Ross~\cite{CR} and Ross/Shapiro \cite{RS}.) 

Important applications of pseudocontinuability are contained in the work of D.Z. Arov on Darlington synthesis,
$J$-inner functions, and related topics (see, e.g., \cite{0} and \cite{0a}).

A further domain of application of pseudocontinuable functions is rational approximation. 
This theme which originated from a series of papers by G.Ts. Tumarkin (see \cite{16a}, \cite{16b}, and \cite{16c})
was treated by Katsnelson in \cite{Ka}. The paper \cite{Ka}
is recommended for some other reasons too. It contains an extensive historical overview on the investigation of
pseudocontinuable functions which also takes into account matrix-valued functions.
Moreover, the paper \cite{Ka} is very well written from the pedagogical point of view. It strongly influenced our
approach to introducing pseudocontinuable functions (see Section~\ref{s2}).

The present paper continues recent investigations on several questions of pseudocontinuability of Schur functions 
in the unit disk (see \cite{BDFK05} and \cite{Dub06}). Our main aim is to establish a direct relation between
two well-known criteria for pseudocontinuability of non-inner Schur functions in the unit disk (see Theorem~\ref{t3.9} 
and Theorem \ref{t4.2}). The main result of this paper is a new characterization of pseudocontinuability of a non--inner Schur function $\Theta$ in the unit disk.
This characterization is expressed in terms of a special orthogonal basis in an
$L_{\mu}^2$ space on the unit circle, where $\mu$ is some probability measure
on the unit circle which is canonically associated with $\Theta$ (see Theorem \ref{4.3}).

The paper is organized as follows. 
In Section~\ref{s1}, we recall that the set of Schur functions in the unit disk stands in bijective correspondences to
the set of normalized Carath\'eodory functions in the unit disk and to the set of probabi\-li\-ty measures on the Borel 
$\sigma$-algebra of the unit circle.
This will give us the possibility to study the pseudocontinuability of a given Schur function in terms of the associated 
normalized Carath\'eodory function and in terms of the associated probability measure, respectively.

In Section~\ref{s2} (influenced by Katsnelson \cite{Ka}), we summarize essential facts on meromorphic functions 
and the concept of pseudocontinuability which is due to H.S. Shapiro.

Section~\ref{s3} contains a function--theoretic approach to the study of pseudocontinuability of non-inner Schur functions 
in the unit disk. We will recognize that a function of this class admits a pseudocontinuation if and only if the associated 
probability measure satisfies the Szeg\H{o} condition and the corresponding Szeg\H{o} function is pseudocontinuable (see Theorem~\ref{t3.9}).

The central theme of Section~\ref{s4} is an operator--theoretic approach to the investigation of pseudocontinuability of 
non-inner Schur functions in the unit disk. The starting point there is the observation that an arbitrary Schur function
can be re\-pre\-sented as a characteristic function of some completely non--unitary contraction in a separable complex Hilbert space.
Then the pseudocontinuability of a non-inner Schur function $\Theta$ can be characterized in terms of the contraction $T$
having $\Theta$ as its characteristic function. More precisely, the maximal unilateral shifts $V_T$ and $V_{T^\ast}$
contained in $T$ and $T^\ast$, respectively, have to fulfill certain conditions of mutual interrelations (see Theorem~\ref{t4.2}). 

In Section~\ref{s5}, we start from a probability measure $\mu$ on the Borel $\sigma$-algebra of the unit circle.
The main aim is to construct a unitary colligation $\Delta_{\mu}$, the characteristic function $\Theta_{\Delta_{\mu}}$ of
which coincides with the Schur function $\Theta_{\mu}$ associated with the measure $\mu$ according to Section~\ref{s1}.

In Section~\ref{s6}, we consider a probability measure $\mu$ on the Borel $\sigma$-algebra of the unit circle for which
the polynomials are non--complete in the space $L^2_{\mu}$. In this case the question arises as to the existence of a 
natural completion of the system of orthonormal polynomials in $L^2_{\mu}$ to a complete orthonormal system in $L^2_{\mu}$.
The primary concern of Section~\ref{s6} is to construct such a natural completion.
In the particular case that the measure $\mu$ is associated with a pseudocontinuable
non-inner Schur function $\Theta$, we will show that the functions obtained by
completing the orthonormal system of polynomials are boundary values of functions
belonging to the meromorphic Nevanlinna class in the unit disk (see Propoisition \ref{p6.14a}).

The goal of Section~\ref{s7} is to reformulate the characterization of the pseudocontinuability of a non-inner Schur 
function given by Theorem~\ref{t4.2} in terms of the complete orthonormal system which was created in Section~\ref{s6} 
(see Theorem \ref{4.3}).  

Finally, Section~\ref{s8} is aimed at determining and further clarifying a direct connection between the criteria of pseudocontinuability 
of a non-inner Schur function which were pointed out in Theorem~\ref{t3.9} and Theorem \ref{t4.2}, respectively.


\section{Interrelated triples consisting of a Schur function, a normalized Carath\'eodory function and a probability measure}
\label{s1}


Let $\mathbb D := \{\zeta\in\mathbb C : |\zeta| < 1\}$ and $\mathbb T := \{t\in\mathbb C : |t| = 1\}$ 
be the unit disk and the unit circle in the complex plane $\mathbb C$, respectively. 
The central object in this paper is the Schur class $\cS (\mathbb D)$ of all functions $\Theta :\mathbb D\to\mathbb C$ 
which are holomorphic in $\mathbb D$ and satisfy $\Theta (\mathbb D)\subseteq \mathbb D\cup\mathbb T$.
Our main aim is to study the phenomenon of pseudocontinuability for functions belonging to $\cS (\mathbb D)$.
In order to allow for a more effective treatment of this question, we first consider certain important objects that relate bijectively to the class $\cS (\mathbb D)$.
This is the main content of the present section.

Let $\Theta\in\cS (\mathbb D)$. Then the function $\Phi:\mathbb D\to\mathbb C$ defined by
\begin{equation}\label{Nr.0.1}
\Phi (\zeta) := \frac{1+\zeta\Theta (\zeta)}{1-\zeta\Theta (\zeta)}
\end{equation}
is holomorphic in $\mathbb D$ and satisfies
\begin{equation}\label{Nr.0.1b}
 {\rm Re}\,[\Phi (\zeta)] > 0,\quad \zeta\in\mathbb D, 
\end{equation}
and
\begin{equation}\label{Nr.0.1c}
 \Phi (0) = 1 . 
\end{equation}
Let $\cC (\mathbb D)$ be the Carath\'eodory class of all functions $\Psi :\mathbb D\to\mathbb C$ 
which are holomorphic in $\mathbb D$ and satisfy ${\rm Re}\,[\Psi(\zeta)] \geq 0$ and let
\begin{equation}\label{Nr.0.1d}
  \cC^0 (\mathbb D) := \{\Psi \in\cC (\mathbb D) :\Psi (0) = 1\}.
\end{equation}
In view of \eqref{Nr.0.1}, \eqref{Nr.0.1b}, \eqref{Nr.0.1c}, and \eqref{Nr.0.1d} we have
\begin{equation}\label{Nr.0.1e}
  \Phi \in \cC^0 (\mathbb D). 
\end{equation}
It can be easily verified that via \eqref{Nr.0.1} a bijective correspondence between the classes $\cS (\mathbb D)$
and $\cC^0 (\mathbb D)$ is established. Note that from \eqref{Nr.0.1} it follows that
\begin{equation}\label{Nr.0.1f}
 \zeta\Theta (\zeta) = \frac{\Phi (\zeta) - 1}{\Phi (\zeta) + 1}, \quad \zeta\in\mathbb D.
\end{equation}

The class $\cC (\mathbb D)$ is intimately related with the class $\cM_+ (\mathbb T)$
of all finite nonne\-ga\-tive measures on the Borel $\sigma$-algebra $\gB$ of $\mathbb T$.
According to the Riesz-Herglotz Theorem (see, e.g., \cite[Theorem 2.2.2]{DFK92}): 
For each function $\Phi \in\cC (\mathbb D)$ there exists a unique measure
$\mu\in\cM_+ (\mathbb T)$ and a unique number $\beta\in\mathbb R$ such that
\begin{equation}\label{Nr.0.3}
 \Phi (\zeta) = \int_\mathbb T \frac{t+\zeta}{t-\zeta}\,\mu (dt) + i\beta,
 \quad \zeta\in\mathbb D.
\end{equation}
Obviously, $\beta = {\rm Im}\, [\Phi (0)]$. On the other hand, it can be easily checked that, for arbitrary
$\mu\in\cM_+ (\mathbb T)$ and $\beta \in\mathbb R$, the function $\Phi$, which is defined by the
right hand side of (\ref{Nr.0.3}), belongs to $\cC (\mathbb D)$. 
If we consider the Riesz-Herglotz representation (\ref{Nr.0.3}) for a function $\Phi\in\cC^0 (\mathbb D)$,
then $\beta =0$ and $\mu$ belongs to the set $\cM_+^1 (\mathbb T)$ of all probability measures
belonging to $\cM_+ (\mathbb T)$ (i.e. $\mu (\mathbb T) = 1)$.
Actually, in this way we obtain a bijective correspondence between the classes $\cC^0 (\mathbb D)$
and $\cM_+^1 (\mathbb T)$.  

In the result of the above considerations we obtain special ordered triples $[\Theta,\Phi,\mu]$ consisting
of a function $\Theta\in\cS (\mathbb D)$, a function $\Phi \in \cC^0 (\mathbb D)$, and a measure 
$\mu\in\cM_+^1 (\mathbb T)$ which are interrelated in such way that each of these three objects uniquely
determines the other two. For that reason, if one of the three objects is given, we will say that the two others are associated with it.
Based on this procedure, in Section~\ref{s3} we will characterize the pseudocontinuability of a function
$\Theta\in\cS (\mathbb D)$ in terms of the associated objects $\Phi \in \cC^0 (\mathbb D)$ and  
$\mu\in\cM_+^1 (\mathbb T)$.


\section{Some basic facts on classes of meromorphic functions and pseudocontinuability}
\label{s2}


In this section, we summarize some facts on several classes of meromorphic functions which will be used later. (A detailed treatment of this subject can be found, e.g., in Duren \cite{Du} and Nevanlinna \cite{N}.)
In particular, we will recall the concept of pseudocontinuation.

We will use $\overline{\mathbb C}$ and $\mathbb E$ to denote the extended complex plane and the exterior of the
closed unit disk, respectively, i.e. $\overline{\mathbb C}:={\mathbb C}\cup\{\infty\}$ and 
$\mathbb E:=\overline{\mathbb C}\setminus (\mathbb D \cup \mathbb T)$.
Furthermore, the symbol $m$ stands for the normalized linear Lebesgue-Borel measure on $\mathbb T$. So we have
$m(\mathbb T)=1$, i.e. $m\in\cM_+^1 (\mathbb T)$.  

Assume that ${\rm G}$ is one of the domains $\mathbb D$ or $\mathbb E$. 
Let $\cN\cM ({\rm G})$ be the Nevanlinna class of all functions which are 
meromorphic in ${\rm G}$ and which can be represented as a quotient of two bounded 
holomorphic functions in ${\rm G}$. If $h\in\cN\cM (\dD)$ (resp. $h\in\cN\cM (\dE)$), then a well-known 
theorem due to Fatou implies that there is a set $B_0\in\gB$ with $m(B_0) = 0$ and a 
Borel measurable function $\underline{h} :\dT\to\dC$ such that
\[ \lim_{r\to 1-0} h(rt) = \underline{h} (t) \quad 
   \Bigl(\mbox{resp.}\, \lim_{r\to 1+0} h(rt) = \underline{h} (t) \Bigr) \]
for all $t\in\dT\setminus B_0$. In the following, we will continue to use the symbol $\underline{h}$
to denote such a boundary function of a function $h$ which belongs to $\cN\cM (\dD)$ or $\cN\cM (\dE)$. 

There is a standard bijective correspondence between the sets $\cN\cM (\dD)$ and $\cN\cM (\dE)$.
In order to describe this we introduce some further notation. Let ${\rm G}$ be a nonempty subset of $\overline{\mathbb C}$.
Then ${\rm G}^{\wh}$ denotes that subset of $\overline{\mathbb C}$ which is symmetric to ${\rm G}$ with respect
to the unit circle, i.e.
\[ {\rm G}^{\wh}:= \left\{ z\in\overline{\mathbb C} : \frac{1}{z^{\ast}}\in {\rm G} \right\} \]
with the usual conventions $\frac{1}{0^{\ast}}:=\infty$ and $\frac{1}{\infty^{\ast}}:=0$. 
Here $z^{\ast}$ means the complex conjugate of $z$.
If $f:{\rm G}\to\dC$, then $f^{\wh}$ stands for the complex-valued function which is defined on ${\rm G}^{\wh}$ by
\[  f^{\wh}(z):= \left(f\Bigl(\frac{1}{z^{\ast}}\Bigr)\right)^{\ast} .  \]
Now the bijective correspondence between the sets $\cN\cM (\dD)$ and $\cN\cM (\dE)$ can be expressed as follows.

\begin{rem} \label{r2.1}
 If $h\in\cN\cM (\dD)$ (resp. $h\in\cN\cM (\dE)$), then the function $h^{\wh}$ belongs to 
 $\cN\cM (\dE)$ (resp. to $\cN\cM (\dD)$) and $\underline{h}^{\ast}$ is the boundary function of $h^{\wh}$.
\end{rem}

\begin{rem} \label{r2.2}
 Let ${\rm G}\in\{\dD,\dT\}$ and let $h\in\cN\cM ({\rm G})$ not vanish identically in ${\rm G}$. Then the function
 $\ln_{\!}|\underline{h}|$ is $m$-integrable and $m(\{\underline{h}=0\})=0$.
\end{rem}

The set $\cN\cM (\dD)$ (resp. $\cN\cM (\dE)$) is obviously an algebra over $\dC$. The subalgebra
of all $h\in\cN\cM (\dD)$ which are holomorphic in $\dD$ will be designated by $\cN (\dD)$.
The class $\cN (\dD)$ can be characterized as the set of all functions $h$ which are holomorphic
in $\dD$ and which satisfy
\begin{equation*}
 \sup_{r\in [0,1)} \int_\dT \ln^{+\!}|h(rt)| \,m(dt) < +\infty ,
\end{equation*}
where $\ln^{+\!}x := \max\{\ln x,0\}$ for each $x\in [0,+\infty)$. 
If a function $D:\dD\to\dC$ admits a representation
\begin{equation*}
 D(\zeta) = \alpha\cdot \exp \left\{
 \int_\dT \frac{t+\zeta}{t-\zeta}  \ln [k(t)] \,m(dt) \right\}, 
 \quad \zeta\in\dD,
\end{equation*}
with some $\alpha\in\dT$ and some Borel measurable function $k:\dT \to [0,+\infty)$ satisfying
\begin{equation*}
 \int_\dT |\ln [k(t)] | \,m(dt) < +\infty,
\end{equation*}
then $D$ belongs to $\cN (\dD)$ and $|\underline{D}|=k$ holds $m$-a.e. on $\dT$. 
Such functions $D$ are called outer functions in $\cN (\dD)$.

\begin{rem} \label{r2.3}
 Let $D$ be an outer function in $\cN (\dD)$. Then $D(\zeta)\not=0$ for all $\zeta\in\dD$ and the function
 $D^{-1}$ is an outer function in $\cN (\dD)$ as well. 
 Moreover, if $D_1$ and $D_2$ are outer functions in $\cN (\dD)$, then $D_1 D_2$ is an outer 
 function in $\cN (\dD)$.
\end{rem}

An outer function $D$ in $\cN (\dD)$ is called normalized if $D(0)\in (0,+\infty)$.

\begin{rem} \label{r2.4}
 Let $\Phi \in \cC(\mathbb D)$ be a function which does not vanish identically in $\dD$. Then $\Phi$
 is an outer function in $\cN (\dD)$.
 If $\Phi \in \cC^0 (\mathbb D)$, then $\Phi$ is a normalized outer function in $\cN (\dD)$.
\end{rem}

In view of Remark~\ref{r2.3} and Remark~\ref{r2.4} one can conclude the following.

\begin{rem} \label{r2.4b}
 Let $\Phi \in \cC(\mathbb D)$. Then $\Phi+1$ is a function belonging to $\cC(\mathbb D)$ which
 vanishes nowhere in $\dD$. Thus, $\Phi+1$ and $(\Phi+1)^{-1}$ are outer functions in $\cN (\dD)$.
\end{rem}

\begin{prop} \label{p2.6}
 Let $w:\dT\to [0,+\infty)$ be an $m$-integrable function which satisfies the Szeg\H{o} condition
\begin{equation}\label{Nr.3.4a}
  \int_\dT \ln [w(t)] \,m(dt)  > -\infty .
\end{equation}
 Then there is a unique normalized outer function $D$ in $\cN (\dD)$ such that the relation 
 $|\underline{D}|^2=w$ holds $m$-a.e. on $\dT$. This function $D:\dD\to\dC$ is given by
\begin{equation}\label{Nr.3.5}
  D(\zeta) = \exp \left\{\frac{1}{2} \int_\mathbb T \frac{t+\zeta}{t-\zeta} \ln [w(t)] \,m(dt)\right\} .
\end{equation}
\end{prop}

The function $D$ which appears in Proposition~\ref{p2.6} is called the Szeg\H{o} function associated
with $w$ (see also Nikishin/Sorokin \cite[Chapter~5 in \S 5]{NS}).

Clearly, the Schur class $\cS(\mathbb D)$ is a subclass of $\cN (\dD)$. A function $I\in\cS(\mathbb D)$
is called an inner function if the relation $|\underline{I}|=1$ holds $m$-a.e. on $\dT$. We denote by $\cI(\mathbb D)$ 
the subclass of $\cS(\mathbb D)$ consisting of all inner functions.

\begin{prop} \label{p2.8}
 {\rm (V. I. Smirnov)} \\
 Let $h\in\cN\cM(\dD)$. Then there exists an outer function $E$ in $\cN (\dD)$ and some functions
 $I_1,I_2\in\cI(\mathbb D)$ such that
 \[ h = E \cdot I_1 \cdot I_2^{-1} . \] 
 In the case of $h\in\cN(\dD)$, $I_2$ can be chosen as constant function with value $1$.
\end{prop}

\begin{rem} \label{r2.9}
 Let $E$ be an outer function in $\cN (\dD)$ which also belongs to $\cI(\mathbb D)$. Then
 $E$ is a constant function with unimodular value. 
\end{rem}

Now we recall the concept of pseudocontinuability. Let $h\in\cN\cM(\dD)$. Then one says that $h$
admits a pseudocontinuation into $\dE$ if there exists a function $h^{\#}\in\cN\cM(\dE)$ such that
the boundary functions $\underline{h}$ and $\underline{h^{\#}}$ coincide $m$-a.e. on $\dT$. From
Remark~\ref{r2.2} it follows that a function $h\in\cN\cM(\dD)$ admits at most one pseudocontinuation into $\dE$.
Note that if $h\in\cN\cM(\dD)$ admits a pseudocontinuation $h^{\#}$ into $\dE$ and if
$h$ is analytically continuable through some open arc of $\dT$, then the analytic continuation
coincides with $h^{\#}$. In the following, the notation $\Pi (\mathbb D)$ stands for the set of all
functions belonging to $\cN\cM(\dD)$ which admit a pseudocontinuation into $\dE$.
Obviously, the restriction of a rational function onto $\dD$ belongs to $\Pi (\mathbb D)$.
If $h\in\Pi (\mathbb D)$, then the symbol $h^{\#}$ will be used to denote the pseudocontinuation of 
$h$ into $\dE$.

Based on Remark~\ref{r2.1} one can draw the following conclusion.

\begin{rem} \label{r2.10}
 If $h\in\Pi (\mathbb D)$, then $(h^{\#})^{\wh}\in\cN\cM(\dD)$ and $\underline{(h^{\#})^{\wh}}=\underline{h}^{\ast}$
 $m$-a.e. on $\dT$. Moreover, if $h\in\cN\cM(\dD)$ such that there is some $H\in\cN\cM(\dD)$ fulfilling the relation
 $\underline{H}=\underline{h}^{\ast}$ $m$-a.e. on $\dT$, then $h\in\Pi (\mathbb D)$ and $h^{\#}=H^{\wh}$. 
\end{rem}

Remark~\ref {r2.10} shows, in particular, that a function $h\in\cN\cM(\dD)$ belongs to $\Pi (\mathbb D)$ if and only if
there exists some $H\in\cN\cM(\dD)$ fulfilling $\underline{H}=\underline{h}^{\ast}$ $m$-a.e. on $\dT$. Thus, the
pseudocontinuability of $h$ is determined by $\underline{h}^{\ast}$.

The following properties of the class $\Pi (\mathbb D)$ can be easily checked.

\begin{rem} \label{r2.11}
 Let $g,h\in\Pi (\mathbb D)$. If $\alpha,\beta\in\dC$, then $\alpha g + \beta h \in \Pi (\mathbb D)$.
 Moreover, the function $gh$ and, in the case that $h$ does not not vanish identically in $\dD$, also
 $h^{-1}$ and $gh^{-1}$ belong to $\Pi (\mathbb D)$.
\end{rem}

A particular situation is met in the class $\cI (\mathbb D)$ of inner functions belonging to $\cS (\mathbb D)$. 
In this case we not only have pseudocontinuability, but can write down the 
corresponding pseudocontinuation, explicitly.

\begin{rem} \label{r2.12}
 The inclusion $\cI (\mathbb D)\subseteq \Pi (\mathbb D)$ holds. Moreover, for some $I\in\cI (\mathbb D)$ 
 the function $I^{\wh}$ does not vanish identically in $\dE$ and $I^{\#}=(I^{\wh\,})^{-1}$. 
\end{rem}

A combination of Remarks \ref{r2.11} and \ref{r2.12} yields the following.

\begin{rem} \label{r2.13}
 Let $h\in\cN\cM(\dD)$ and let the outer function $E$ in $\cN(\dD)$ be chosen according to Proposition \ref{p2.8}.
 Then $h\in\Pi (\mathbb D)$ if and only if $E\in\Pi (\mathbb D)$.
\end{rem}

Remark~\ref{r2.13} indicates the importance of the study of pseudocontinuability of outer functions in $\cN(\dD)$.
This phenomenon can be characterized in the following way in terms of inner functions belonging to $\cS(\mathbb D)$.

\begin{prop} \label{p2.14}
 Let $D$ be an outer function in $\cN(\dD)$. Then $D\in\Pi (\mathbb D)$ if and only if there exist functions
 $I_1,I_2\in\cI(\mathbb D)$ such that the identity
\begin{equation} \label{Nr.p1}
  \underline{D} \cdot (\underline{D}^{\ast})^{-1} =  \underline{I_2} \cdot \underline{I_1}^{-1} 
\end{equation}
 is satisfied $m$-a.e. on $\dT$.
\end{prop}

\begin{proof}
Let $D\in\Pi (\mathbb D)$. Then the pseudocontinuation $D^{\#}$ of $D$ into $\dE$ is well-defined 
and we set $H:=(D^{\#})^{\wh}$. From Remark~\ref{r2.10} we get $H\in\cN\cM(\dD)$ and
\begin{equation} \label{Nr.p2}
 \underline{H} = \underline{D}^{\ast} .
\end{equation}
Because of $H\in\cN\cM(\dD)$ the canonical factorization theorem of Smirnov (see Proposition~\ref{p2.8}) 
provides us with the existence of an outer function $E$ in $\cN(\dD)$ and two functions $J_1,J_2\in\cI(\mathbb D)$ such that
\begin{equation} \label{Nr.p3}
  H = E \cdot J_1 \cdot J_2^{-1} . 
\end{equation}
In view of Remark~\ref{r2.3} the function $E$ vanishes nowhere in $\dD$ and $F:=D E^{-1}$ is an outer
function in $\cN(\dD)$. Because of \eqref{Nr.p3} we obtain
\[ F = D \cdot H^{-1} \cdot J_1 \cdot J_2^{-1} . \]
Thus, using \eqref{Nr.p2} we infer
\begin{equation} \label{Nr.p4}
 \underline{F} = \underline{D} \cdot (\underline{D}^{\ast})^{-1} \cdot \underline{J_1} \cdot \underline{J_2}^{-1}  
\end{equation}
$m$-a.e. on $\dT$. Since $\underline{J_1}$ and $\underline{J_2}$ are both unimodular $m$-a.e. on $\dT$ from  
\eqref{Nr.p4} one can see that $\underline{F}$ is unimodular $m$-a.e. on $\dT$. Furthermore, since we already know
that $F$ is an outer function in $\cN(\dD)$ the maximum modulus principle of Smirnov 
(see, e.g., Duren \cite[Theorem~2.11 on page 25]{Du})
implies that $F$ belongs to $\cI(\mathbb D)$ as well. Hence, Remark~\ref{r2.9} shows that $F$ is a constant function with
unimodular value. Consequently, the settings $I_2:=F J_2$ and $I_1:=J_1$ yield functions belonging to 
$\cI(\mathbb D)$ such that \eqref{Nr.p1} holds $m$-a.e. on $\dT$.
Conversely, we assume that there exist functions $I_1,I_2\in\cI(\mathbb D)$ such that \eqref{Nr.p1} 
is satisfied $m$-a.e. on $\dT$. Let $H:=D I_1 I_2^{-1}$. Then $H\in\cN\cM(\dD)$ and in view of \eqref{Nr.p1} we get
\[ \underline{H} = \underline{D}^{\ast} . \]
This leads, in combination with Remark~\ref{r2.10}, to $D\in\Pi (\mathbb D)$.
\end{proof}

It should be mentioned that in the above proof of the fact that the validity of \eqref{Nr.p1} $m$-a.e. on $\dT$
implies $D\in\Pi (\mathbb D)$ we do not use the assumption that $D$ is an outer function in $\cN(\dD)$.


\section{A function--theoretic approach to study the pseudo\-continuability in $\cS(\mathbb D)\setminus\cI(\mathbb D)$}
\label{s3}


In the present section we investigate the pseudocontinuability of a non-inner function $\Theta\in\cS(\mathbb D)$ in terms
of the function $\Phi \in \cC^0 (\mathbb D)$ and the measure $\mu\in\cM_+^1 (\mathbb T)$ which is generated by 
$\Theta$ as explained in Section~\ref{s1}.  
The following result provides us with further insight into the connection between functions $\Theta$ and $\Phi$.

\begin{prop} \label{p3.1}
 Let $\Theta\in\cS(\mathbb D)$ and let $\Phi \in \cC^0 (\mathbb D)$ be defined via \eqref{Nr.0.1}.
 Then $\Theta\in\Pi (\mathbb D)$ if and only if $\Phi\in\Pi (\mathbb D)$.
\end{prop}

\begin{proof}
Taking into account \eqref{Nr.0.1} and \eqref{Nr.0.1f} an application of Remark~\ref{r2.11} yields the asserted equivalence.
\end{proof}

The rest of this section is devoted to the study of pseudocontinuability of a non-inner function $\Theta\in\cS(\mathbb D)$ 
in terms of its associated measure $\mu\in\cM_+^1 (\mathbb T)$.
Our approach to that question will be based on an analysis of various aspects of the boundary behavior of $\Theta$. 
In particular, the $m$-inte\-gra\-bi\-li\-ty of the function $\ln [1-|\underline{\Theta}|^2]$ will be of importance.

\begin{prop}\label{p3.2}
 Let $\Theta\in\cS(\mathbb D)\cap\Pi (\mathbb D)$ and let $h:=1-\Theta(\Theta^{\#})^{\wh}$. Then:
\begin{enumerate}
\item[(a)] 
 The function $h$ belongs to $\cN(\dD)$ and satisfies $\underline{h} = 1-|\underline{\Theta}|^2$ 
 $m$-a.e. on $\dT$.
\item[(b)] 
 If $\Theta\in\cS(\mathbb D) \setminus \cI(\mathbb D)$, then the function $\ln_{\!}\underline{h}$ is
 $m$-integrable.
\end{enumerate}
\end{prop}

\begin{proof}
Using some elementary properties on the class $\cN\cM(\dD)$ with a view to Remark~\ref{r2.10}, assertion
(a) follows. Now let $\Theta\in\cS(\mathbb D) \setminus \cI(\mathbb D)$. Because of (a) one can see that the function
$\underline{h}$ does not vanish $m$-a.e. on $\dT$. Therefore, the function $h$ does not vanish identically in $\dD$.
Consequently, Remark~\ref{r2.2} provides finally assertion (b). 
\end{proof}

Let $\Theta\in\cS(\mathbb D)$ be embedded in the triple $[\Theta,\Phi,\mu]$ as in Section~\ref{s1}.
Motivated by Proposition~\ref{p3.2} our next considerations are aimed at characterizing the $m$-integra\-bi\-li\-ty
of the function $\ln [1-|\underline{\Theta}|^2]$ in terms of the function $\Phi$ or the measure $\mu$.

\begin{rem} \label{r3.3}
 Let $\Theta\in\cS(\mathbb D)$ and let $\Phi \in \cC^0 (\mathbb D)$ be defined by \eqref{Nr.0.1}.
 Then a straightforward calculation yields that the identity
\[ 1-|\underline{\Theta}(t)|^2 = \frac{4{\rm Re}\,[\underline{\Phi}(t)]}{|\underline{\Phi}(t)+1|^2} \]
 holds $m$-a.e. on $\dT$.
\end{rem}

\begin{prop} \label{p3.4}
 Let $\Theta\in\cS(\mathbb D)$ and let $\Phi \in \cC^0 (\mathbb D)$ be defined by \eqref{Nr.0.1}.
 Then the function $\ln [1-|\underline{\Theta}|^2]$ is $m$-integrable if and only if 
 $\ln [{\rm Re}\,\underline{\Phi}]$ is $m$-integrable.
\end{prop}

\begin{proof}
Because of Remark~\ref{r3.3} we get
\begin{equation} \label{Nr.31}
 \ln [1-|\underline{\Theta}|^2]
 = \ln 4 + \ln [{\rm Re}\,\underline{\Phi}] - 2 \ln_{\!}|\underline{\Phi}+1| 
\end{equation}
$m$-a.e. on $\dT$. 
In view of Remark~\ref{r2.4b} we know that $\Phi+1$ is an outer function in $\cN(\mathbb D)$. Thus, 
Remark~\ref{r2.3} implies in combination with Remark~\ref{r2.2} that 
$\ln_{\!}|\underline{\Phi}+1|$
is $m$-integrable. Hence, the assertion is an immediate consequence of \eqref{Nr.31}.
\end{proof}

Following D.Z. Arov we denote by $\cS\Pi(\mathbb D)$ the subclass of pseudocontinuable functions
of $\cS(\mathbb D)$.

\begin{cor} \label{c3.5}
 Let $\Theta\in \cS\Pi(\mathbb D)\setminus \cI(\mathbb D)$ and let 
 $\Phi \in \cC^0 (\mathbb D)$ be defined by \eqref{Nr.0.1}.
 Then the function $\ln [{\rm Re}\,\underline{\Phi}]$ is $m$-integrable.
\end{cor}

\begin{proof}
Combine part~(b) of Proposition~\ref{p3.2} with Proposition~\ref{p3.4}.
\end{proof}

Now we start from a function $\Phi \in \cC (\mathbb D)$ and look to characterize the 
$m$-integrability of $\ln [{\rm Re}\,\underline{\Phi}]$ in terms of 
its Riesz-Herglotz measure $\mu$ associated with $\Phi$, subject to \eqref{Nr.0.3}.
In order to realize this goal we will apply the following result which is an immediate consequence 
of a theorem due to Fatou on the boundary behavior of Poisson integrals 
(see, e.g., Rosenblum/Rovnyak \cite[Theorem~1.18]{RR}).

\begin{prop} \label{p3.6}
 Let $\Phi \in \cC (\mathbb D)$ and let $\mu\in\cM_{+}(\dT)$ be associated with $\Phi$ by the 
 Riesz-Herglotz Theorem, i.e. via \eqref{Nr.0.3} with $\beta={\rm Im}\,[\Phi (0)]$.
 Furthermore, let the Lebesgue decomposition of $\mu$ with respect to $m$ be given by 
\begin{equation}\label{Nr.3.4}
  \mu (dt) = w(t) m(dt) + \mu_s (dt) ,
\end{equation}
 where $\mu_s$ stands for the singular part of $\mu$ with respect to $m$. Then the relation
 ${\rm Re}\,\underline{\Phi}=w$ holds $m$-a.e. on $\dT$.
\end{prop}

Proposition~\ref{p3.6} leads us to a particular subclass of 
$\cM_{+}(\dT)$.

Let $\mu\in\cM_{+}(\dT)$ and let the Lebesgue decomposition of $\mu$ with respect to $m$ be given by
\eqref{Nr.3.4}, where $\mu_s$ stands for the singular part of $\mu$ with respect to $m$.
Then $\mu$ is said to satisfy the Szeg\H{o} condition if the function $\ln w$ is $m$-integrable
or equivalently if \eqref{Nr.3.4a} holds. In this case, the Szeg\H{o} function 
$D:\mathbb D\to\mathbb C$ associated with $w$ is given by \eqref{Nr.3.5}.
If we start from a measure $\mu\in\cM_{+}(\dT)$ which satisfies the Szeg\H{o} condition, we 
will call $D$ also the Szeg\H{o} function associated with $\mu$. 

An application of Proposition~\ref{p3.6} immediately provides the following result. 

\begin{cor} \label{c3.7}
 Let $\Phi\in\cC (\mathbb D)$ and let $\mu\in\cM_{+}(\dT)$ be associated with $\Phi$ by the 
 Riesz-Herglotz theorem via \eqref{Nr.0.3} with $\beta={\rm Im}\,[\Phi (0)]$. 
 Then the function $\ln [{\rm Re}\,\underline{\Phi}]$ is $m$-integrable if and only if 
 the measure $\mu$ satisfies the Szeg\H{o} condition.
\end{cor}

\begin{lem} \label{l3.8}
 Let $\Phi$ and $D$ be functions belonging to $\cN\cM (\dD)$ which do not vanish identically in $\dD$,
 where the relation ${\rm Re}\,\underline{\Phi}=|\underline{D}|^2$ is satisfied $m$-a.e. on $\dT$.
 Then $\Phi\in\Pi (\mathbb D)$ if and only if $D\in\Pi (\mathbb D)$.
\end{lem}

\begin{proof}
First suppose $\Phi\in\Pi (\mathbb D)$. Then we define
\[ H := \frac{1}{2} [\Phi + (\Phi^{\#})^{\wh\,}] \cdot D^{-1} \]
In view of the choice of $\Phi$ and $D$, Remark~\ref{r2.1}, and the fact that the set $\cN\cM (\dD)$ 
is an algebra over $\dC$ it follows that $H\in\cN\cM (\dD)$. Moreover, an application of Remark~\ref{r2.10} 
yields the equality $\underline{(\Phi^{\#})^{\wh\,}}=\underline{\Phi}^{\ast}$ $m$-a.e. on $\dT$. 
Hence, $\underline{H}= \underline{D}^{\ast}$ $m$-a.e. on $\dT$. 
Taking this and $H\in\cN\cM (\dD)$ into account from Remark~\ref{r2.10} we get $D\in\Pi (\mathbb D)$.
Conversely, we suppose now that $D\in\Pi (\mathbb D)$. We then set
\[ G := 2 D \cdot (D^{\#})^{\wh} - \Phi \]
Similar to the above, using Remark~\ref{r2.1}, one can reason that $G\in\cN\cM (\dD)$, where 
Remark~\ref{r2.10} results firstly in $\underline{G}= \underline{\Phi}^{\ast}$ $m$-a.e. on $\dT$. Thus, from 
Remark~\ref{r2.10} we can conclude $\Phi\in\Pi (\mathbb D)$.
\end{proof}

The following criterion for the pseudocontinuability of a non-inner function $\Theta\in\cS(\mathbb D)$ 
is the main result of this section. It is formulated in terms of the pro\-ba\-bility measure
$\mu\in\cM_{+}^{1}(\dT)$ which is associated with $\Theta$ as explained in Section~\ref{s1}.

\begin{thm} \label{t3.9}
 Let $\Theta\in\cS(\mathbb D)\setminus\cI(\mathbb D)$ and let $\mu\in\cM_{+}^{1}(\dT)$ be the measure
 associated with $\Theta$ according to {\rm Section~\ref{s1}}. Then the following statements are equivalent:
\begin{enumerate}
\item[(i)] 
 $\Theta\in\Pi(\mathbb D)$.
\item[(ii)] 
 The measure $\mu$ satisfies the Szeg\H{o} condition and the Szeg\H{o} function $D$ associated with
 $\mu$ belongs to $\Pi(\mathbb D)$.
\end{enumerate}
\end{thm}

\begin{proof}
Let $\Phi \in \cC^0 (\mathbb D)$ be defined by \eqref{Nr.0.1}.
Assume that the measure $\mu$ satisfies the Szeg\H{o} condition. Hence, from Proposition~\ref{p2.6} we infer
that the Szeg\H{o} function $D$ associated with $\mu$ belongs to $\cN(\mathbb D)$. Moreover, the combination
of Proposition~\ref{p2.6} and Proposition~\ref{p3.6} yields the equality
\begin{equation} \label{Nr.x1}
 {\rm Re}\,\underline{\Phi} = |\underline{D}|^2 
\end{equation}
$m$-a.e. on $\dT$. Suppose now that (i) holds. 
Then from Proposition~\ref{p3.1} we get 
\begin{equation} \label{Nr.x2}
 \Phi\in\Pi(\mathbb D).
\end{equation}
Moreover, in view of (i) and Corollary~\ref{c3.5}, the function $\ln [{\rm Re}\,\underline{\Phi}]$ is $m$-integrable.
Thus, Corollary~\ref{c3.7} shows that $\mu$ satisfies the Szeg\H{o} condition. Consequently, \eqref{Nr.x1} is
satisfied. By virtue of \eqref{Nr.x1}, \eqref{Nr.x2}, and Lemma~\ref{l3.8} it follows that $D\in\Pi(\mathbb D)$.
So we have verified that (i) implies (ii). Conversely, we suppose now that (ii) holds. From (ii) we obtain 
\eqref{Nr.x1} and $D\in\Pi(\mathbb D)$. Therefore, Lemma~\ref{l3.8} yields \eqref{Nr.x2}. Accordingly, an 
application of Proposition~\ref{p3.1} supplies $\Theta\in\Pi(\mathbb D)$, i.e. (i).
\end{proof}


\section{An~operator--theoretic~approach~to~study~the~pseudocontinuability in $\cS(\mathbb D)\setminus\cI(\mathbb D)$}
\label{s4}


The starting point of this section is the observation that a given Schur function $\Theta\in\cS(\mathbb D)$ can be
represented as characteristic function of some contraction in a Hilbert space. That means that there exists a separable
complex Hilbert space and bounded linear operators $T:\gH\to\gH$,\, $F:\dC\to\gH$,\, $G:\gH\to\dC$, and $S:\dC\to\dC$ 
such that the block operator 
\begin{equation}\label{Nr.d1.2b}
 U:= \left(\begin{matrix} T & F \cr G & S \end{matrix}\right) :\gH\oplus\dC\to\gH\oplus\dC
\end{equation}
is unitary and moreover that the equality
\begin{equation}\label{Nr.d1.2c}
 \Theta (\zeta) = S + \zeta G (I - \zeta T)^{-1} F,
 \quad \zeta\in\mathbb D,
\end{equation}
is fulfilled. Note that in \eqref{Nr.d1.2b} the complex plane $\dC$ is considered as the one-dimensional
complex Hilbert space with the usual inner product
\[ \bigl(z,w\bigr)_{\dC} = z^{\ast}w,
\quad z,w\in\dC.   \]

The unitarity of the operator $U$ implies that the operator $T$ is contractive (i.e. $\|T\|\leq 1$).
Thus, for all $\zeta\in\dD$ the operator $I - \zeta T$ is boundedly invertible. 
The unitarity of the operator $U$ means that the ordered tuple
\begin{equation}\label{Nr.d1.5b}
 \bigtriangleup = (\gH, \dC, \dC; T, F, G, S)
\end{equation}
is a unitary colligation. In view of \eqref{Nr.d1.2c} the function $\Theta$ is the characteristic
operator function of the unitary colligation $\bigtriangleup$. 

The following subspaces of $\gH$ will later play an important role
\begin{equation}\label{Nr.d1.5c} 
 \gH_{\gF}:=\bigvee\limits_{n=0}^{\infty}T^nF(\dC),\quad
 \gH_{\gG}:=\bigvee\limits_{n=0}^{\infty}(T^{*})^{n} G^*(\dC) .
\end{equation}
By the symbol $\bigvee_{n=0}^{\infty} A_n$ we mean the smallest closed subspace 
generated by the subsets $A_n$ of the corresponding vector spaces. 
The spaces $\gH_\gF$ and $\gH_\gG$ are called the subspaces of controllability and
observability, respectively.
We note that the unitary operator $U$ can be chosen such that 
\begin{equation}\label{Nr.d1.5d} 
  \gH = \gH_\gF\vee\gH_\gG 
\end{equation}
holds. In this case the unitary colligation $\bigtriangleup$ is called simple. 
The simplicity of a unitary colligation means that there does not exist a nontrivial 
invariant subspace of $\gH$ on which the operator $T$ induces a unitary operator.
Such contractions, $T$, are called completely non--unitary.

In the language of unitary colligations the pseudocontinuability of a non-inner Schur function 
$\Theta\in\cS(\mathbb D)$ has the following consequence.

\begin{prop} \label{p4.1}
 Let $\Theta\in \cS\Pi(\mathbb D)\setminus \cI(\mathbb D)$ and let 
 $\bigtriangleup$ be a simple unitary colligation, \eqref{Nr.d1.5b} the characteristic
 operator function of which coincides with $\Theta$. Then the subspaces
\begin{equation}\label{Nr.d1.5e} 
 \gH_{\gF}^\perp:=\gH\ominus\gH_{\gF},\quad  \gH_{\gG}^\perp:=\gH\ominus\gH_{\gG}
\end{equation}
 are nontrivial.
\end{prop}

Because of \eqref{Nr.d1.5c} and \eqref{Nr.d1.5e} it follows that the subspace $\gH_\gG^\perp$ 
$($resp. $\gH_\gF^\perp)$ is invariant with respect to $T$ $($resp. $T^*)$.
It can be shown (see \cite[Chapter~1]{Dub06}) that 
\[  V_T:= {\rm Rstr}._{\gH_\gG^\perp}T, \quad
    V_{T^*}:= {\rm Rstr.}_{\gH_\gF^\perp}T^* \]
are unilateral shifts. More precisely, $V_T$ (resp. $V_{T^*}$) is just the maximal unilateral shift
contained in $T$ $($resp. $T^*)$. This means that an arbitrary invariant subspace with
respect to $T$ $($resp. $T^*)$ on which $T$ $($resp. $T^*)$ induces a unilateral shift is 
contained in $\gH_\gG^\perp$ $($resp. $\gH_\gF^\perp)$. 

In terms of unitary colligations the pseudocontinuability of a non-inner Schur function 
$\Theta\in\cS(\mathbb D)$ was characterized in \cite[Theorem~3.17]{BDFK05} as follows.

\begin{thm} \label{t4.2}
 Let $\Theta\in\cS(\mathbb D)$ and let $\bigtriangleup$ be a simple unitary colligation \eqref{Nr.d1.5b},
 the characteristic operator function of which coincides with $\Theta$. Then the following
 statements are equivalent:
\begin{enumerate}
\item[(i)] 
 $\Theta\in\Pi(\mathbb D)\setminus\cI(\mathbb D)$.
\item[(ii)] \vspace{0.5mm}
 $\gH_{\gG}\cap\gH_\gF^\perp \not= \{0\}$.
\item[(iii)] \vspace{0.5mm}
 $\gH_{\gF}\cap\gH_\gG^\perp \not= \{0\}$.
\end{enumerate}
\end{thm}

A comparison of Theorem~\ref{t3.9} and Theorem~\ref{t4.2} shows that they contain
rather different characterizations of the pseudocontinuability of a non-inner Schur function  
$\Theta\in\cS(\mathbb D)$.
Our subsequent considerations are aimed at establishing direct connections between the two different 
criteria.


\section{A unitary colligation associated with a Borel probability measure on the unit circle}
\label{s5}


Let $\mu\in\cM^1_+ (\mathbb T)$. Then our subsequent
considerations are concerned with the investigation of the unitary
operator $U_\mu^\times$ which is defined for $f\in L_\mu^2$ by
\begin{equation}\label{Nr.2.0}
 (U_\mu^\times f) (t) := \frac{1}{t} \cdot f(t),
 \quad t\in\mathbb T.
\end{equation}
Denote by $\tau$ the embedding operator of $\mathbb C$ into $L_\mu^2$, i.e. $\tau :\mathbb C\to L_\mu^2$
and for each $c\in\mathbb C$ the image $\tau (c)$ of $c$ is the constant function on $\mathbb T$
with value $c$. Denote by $\mathbb C_\mathbb T$ the subspace of $L_\mu^2$ which is generated by
the constant functions and denote by ${\bf 1}$ the constant function on $\mathbb T$ with value $1$.
Then obviously $\tau (\mathbb C) =\mathbb C_\mathbb T$ and $\tau (1) = {\bf 1}$.

We consider the subspace
\[
\gH_\mu := L_\mu^2 \ominus \mathbb C_\mathbb T.
\]
Denote by
\[
U_\mu^\times = \left(\begin{matrix}
T^\times & F^\times\cr
G^\times & S^\times
\end{matrix}\right)
\]
the block representation of the operator $U_\mu^\times$ with respect to the orthogonal decomposition
$L_\mu^2 = \gH_\mu\oplus \mathbb C_\mathbb T$. 
Then (see \cite[Section~2.8]{Dub06}) the following result holds.

\begin{thm}\label{2.1}
Let $\mu\in\cM_+^1 (\mathbb T)$. Define $T_\mu := T^\times$,\; $F_\mu := F^\times \tau$,\;
$G_\mu := \tau^\ast G^\times$,\;and $S_\mu := \tau^\ast S^\times\tau$. Then
\begin{equation}\label{Nr.2.2}
\bigtriangleup_\mu := ( \gH_\mu,\mathbb C, \mathbb C;  T_\mu, F_\mu, G_\mu, S_\mu)
\end{equation}
is a simple unitary colligation, the characteristic function $\Theta_{\bigtriangleup_\mu}$ of which
coincides with the Schur function $\Theta$ associated with $\mu$.
\end{thm}

In view of Theorem~\ref{2.1} the operator $T_\mu$ is a completely non--unitary contraction and if the function $\Phi$ is given by 
(\ref{Nr.0.3}) with $\beta=0$, then from (\ref{Nr.0.1f}) it follows that
\[ \zeta\Theta_{\bigtriangleup_\mu} (\zeta) = \frac{\Phi (\zeta) -1}{\Phi (\zeta) + 1},
\quad\zeta\in\mathbb D. \]

\begin{defn}
 Let $\mu\in\cM_+^1 (\mathbb T)$. Then the simple unitary colligation given by \eqref{Nr.2.2} is called the
 unitary colligation associated with $\mu$.
\end{defn}


\section{Completion of the system of orthogonal polynomials in the space $L_\mu^2$ in the case of non--completeness}
\label{s6}


Let $\mu\in\cM_+^1 (\mathbb T)$ and suppose that the measure $\mu$ has infinitely many points of growth. 
Furthermore, we use, for all integers $n$, the mapping $e_n :\mathbb T\to\mathbb C$ defined by
\begin{equation}\label{Nr.EK}
 e_n (t) := t^n .
\end{equation}
Thus, we have $e_{-n} = (U_\mu^\times)^n {\bf 1}$,
where $U_\mu^\times$ is the operator defined by (\ref{Nr.2.0}). 
We now consider the system $\{e_0, e_{-1}, e_{-2}, \ldots\}$. 
Applying the Gram-Schmidt orthogonalization method in the space $L_\mu^2$ we get a unique sequence
$(\varphi_n)_{n=0}^\infty$ of polynomials, where 
\begin{equation}\label{Nr.3.1}
 \varphi_n (t) = \alpha_{n,n} t^{-n} + \alpha_{n,n-1} t^{-(n-1)} + \cdots + \alpha_{n,0},
 \; t\in\mathbb T,\quad n\in \{0,1,2,\ldots \},
\end{equation}
such that the conditions
\begin{equation}\label{Nr.3.2}
 \bigvee_{k=0}^n \varphi_k = \bigvee_{k=0}^n (U_\mu^\times)^k {\bf 1},
 \quad \bigl((U_\mu^\times)^n {\bf 1}, \varphi_n\bigr)_{L_\mu^2} > 0, 
 \quad n\in \{0,1,2,\ldots \},
\end{equation}
are satisfied. Here and in the following, if $(h_\alpha)_{\alpha\in A}$ is some family of elements of $L_\mu^2$, then the
symbol $\bigvee_{\alpha\in A} h_\alpha$ stands for the smallest closed subspace in $L_\mu^2$ which
contains all elements of this family.
We note that the second condition in (\ref{Nr.3.2}) is equivalent to
$\bigl( {\bf 1}, \varphi_0 \bigr)_{L_\mu^2} > 0$ and 
\begin{equation}\label{Nr.3.3}
 \bigl(U_\mu^\times \varphi_{n-1},\varphi_n\bigr)_{L_\mu^2} > 0, 
 \quad n\in \{1,2,\ldots \}.
\end{equation}
In particular, since $\mu (\mathbb T)=1$ from the construction of $\varphi_0$ we see that
\begin{equation}\label{Nr.4.11}
 \varphi_0 = \boldsymbol{1}.
\end{equation}

We consider the case that the system $(\varphi_n)_{n=0}^\infty$ is non--complete in the space
$L_\mu^2$. In this case the question arises as to the existence of a natural completion of this
system to a complete orthonormal system in $L_\mu^2$. The main goal of this section is to construct
such a natural completion.

With a view to the Lebesgue decomposition \eqref{Nr.3.4}, where $\mu_s$ stands for the singular part of $\mu$ 
with respect to $m$, it is well known (see, e.g., Rosenblum/Rovnyak \cite[Chapter~4]{RR}) 
that the system of polynomials (\ref{Nr.3.1}) is non--complete in $L_\mu^2$ if and only if 
the Szeg\H{o} condition \eqref{Nr.3.4a} is satisfied. 
In particular, in the case studied below, the Szeg\H{o} function $D:\mathbb D\to\mathbb C$ 
which is given by \eqref{Nr.3.5} is well-defined.

\begin{rem}\label{3.1}
Denote by $E_{\mu_s}$ a support of the measure $\mu_s$ on $\mathbb T$ and denote by $E'$ the set of all
points of $\mathbb T$, where the function $D$ does not have non--tangential boundary values.
Let $E:= E_{\mu_s}\cup E'$. Then $m(E) = 0$. Thus $($see Hoffman \cite[Exercise~3 in Chapter~4]{H}$)$,
we can choose a boundary function $\underline{D}$ of $D$ which satisfies
\[
\underline{D} (t) = 0,\quad t\in E,
\vspace{-2mm} \]
and \vspace{-2mm}
\[
\underline{D} (t)\ne 0,\quad t\in\mathbb T\setminus E.
\]
\end{rem}

To construct the above-mentioned completion we consider a simple unitary colligation
$\bigtriangleup_\mu$ of type (\ref{Nr.2.2}) associated with the measure $\mu$. Here the
completion of the orthonormal system (\ref{Nr.3.1}) of polynomials to a complete orthonormal system
in $L_\mu^2$ will be determined in a natural way by properties of the operator $T_\mu$.

We note that the controllability space (cf. \eqref{Nr.d1.5c}) associated with the unitary colligation 
$\bigtriangleup_\mu$ has the form
\begin{equation*}
\gH_{\mu,\gF} = \bigvee_{n=0}^\infty (T_\mu)^n F_\mu (1).
\end{equation*}
Let the sequence of functions $(\varphi'_k)_{k=1}^\infty$ be defined by
\begin{equation}\label{Nr.3.9}
\varphi'_k := T_\mu^{k-1}  F_\mu (1), \quad k\in\{1,2,\ldots \}.
\end{equation}
In view of the formula
\begin{equation}\label{Nr.3.8}
 \bigvee_{k=0}^n (U_\mu^\times)^k {\bf 1} 
 = \left( \bigvee_{k=0}^{n-1} (T_\mu)^k F_\mu (1)\right) \oplus \mathbb C_\mathbb T, 
 \quad n\in\{1,2,\ldots \},
\end{equation}
it can be seen that the sequence $(\varphi_k)_{k=1}^\infty$ can
be obtained by applying the Gram-Schmidt orthonormalization procedure to 
$(\varphi'_k)_{k=1}^\infty$ with additional consideration of the normalization 
condition (\ref{Nr.3.3}). Thus, we obtain the following result.

\begin{thm}\label{3.2}
 The system $(\varphi_k)_{k=1}^\infty$ of orthonormal polynomials is a basis in the space 
 $\gH_{\mu, \gF}$. This system can be obtained by applying the Gram-Schmidt
 orthogonalization procedure to the sequence \eqref{Nr.3.9}, taking into account the normalization 
 condition \eqref{Nr.3.3}.
\end{thm}

\begin{cor}\label{3.3}
 The orthonormal system of polynomials $(\varphi_k)_{k=0}^\infty$ is non--complete in $L_\mu^2$  
 if and only if $\gH_\mu\ominus \gH_{\mu,\gF}\ne \{0\}$.
\end{cor}

If $T$ is a contraction acting on some Hilbert space $\gH$, then we set
\[ \delta_T:=\dim {\gD_T} \quad 
   \bigl(\mbox{resp.}\; \delta_{T^*}:=\dim {\gD_{T^*}}  \bigr) , \]
where $\gD_T:=\overline{D_T(\gH)}$ (resp. $\gD_{T^*}:=\overline{D_{T^*}(\gH)}_{\,}$) 
is the closure of the range of the defect operator $D_T:=\sqrt{I_\gH-T^*T}$
(resp. $D_{T^*}:=\sqrt{I_\gH-TT^*}_{\,}$).

Let $\gH_{\mu,\gF}^\perp := \gH_\mu\ominus \gH_{\mu,\gF}$ (cf. \eqref{Nr.d1.5e}). 
Then the space $\gH_{\mu,\gF}^\perp$ is invariant with respect to $T_{\mu}^*$ and moreover,
if $\gH_{\mu,\gF}^\perp\not=\{0\}$, then the restriction
\[  V_{T_{\mu}^*}:= {\rm Rstr.}_{\gH_{\mu,\gF}^\perp}T_{\mu}^* \]
is the maximal unilateral shift contained in $T_{\mu}^*$ (see \cite[Theorem~1.6]{Dub06}). 

We now suppose that $\gH_{\mu,\gF}^\perp\not=\{0\}$.
In view of $\delta_{T_\mu} = \delta_{T_\mu^\ast} =1$, 
the multiplicity of the unilateral shift $V_{T_\mu^\ast}$ is equal $1$ and coincides with $\delta_{T_\mu}$.
This is equi\-va\-lent to the operator $T_\mu$ containing a maximal unilateral shift
$V_{T_\mu}$ of multiplicity $1$ (see \cite[Remark~1.11]{Dub06}). Thus, we obtain the following result.

\begin{cor}\label{3.4}
 The orthonormal system of polynomials $(\varphi_k)_{k=0}^\infty$ is non--complete in $L_\mu^2$
 if and only if the contraction $T_\mu$ $($resp. $T_\mu^\ast)$ contains a maximal unilateral 
 shift $V_{T_\mu}$ $($resp. $V_{T_\mu^\ast})$ of multiplicity $1$.
\end{cor}

We consider the orthogonal decomposition
\begin{equation}\label{Nr.67}
 \gH_\mu = \gH_{\mu,\gF} \oplus \gH_{\mu,\gF}^\perp.
\end{equation}
Denote by $\tilde{\gL}_0$ the wandering subspace which generates the subspace associated with the unilateral shift 
$V_{T_\mu^\ast}$. Then $\dim \tilde{\gL}_0 = 1$ and since $V_{T_\mu^\ast}$ is an isometric operator we have
\begin{equation}\label{Nr.3.10}
 V_{T_\mu^\ast} = {\rm Rstr.}_{\gH_{\mu,\gF}^\perp} (U_\mu^\times)^\ast.
\end{equation}
Consequently,
\begin{equation}\label{Nr.3.11}
 \gH_{\mu,\gF}^\perp 
 = \bigoplus\limits_{n=0}^\infty V_{T_\mu^\ast}^n (\tilde{\gL}_0)
 = \bigvee_{n=0}^\infty (T_\mu^\ast)^n (\tilde{\gL}_0)
 = \bigvee_{n=0}^\infty [ (U_\mu^\times)^\ast]^n (\tilde{\gL}_0).
\end{equation}
There exists (see \cite[Corollary~1.10]{Dub06}) a unique unit vector $\psi_1\in\tilde{\gL}_0$ which fulfills
\begin{equation}\label{Nr.3.12}
 \bigl( G_{\mu}^\ast (1), \psi_1 \bigr)_{L_\mu^2} > 0.
\end{equation}
Because of (\ref{Nr.3.10}), (\ref{Nr.3.11}), and (\ref{Nr.3.12}) it follows that the
sequence $(\psi_k)_{k=1}^\infty$, where
\begin{equation}\label{Nr.3.13}
 \psi_k := [(U_\mu^\times)^\ast]^{k-1} \psi_1,
 \quad k\in\{1,2,\ldots\},
\end{equation}
is the unique orthonormal basis of the space $\gH_{\mu,\gF}^\perp$ which satisfies the conditions
\begin{equation}\label{Nr.3.14}
 \bigl( G_{\mu}^\ast (1), \psi_1 \bigr)_{L_\mu^2} > 0,
 \quad \psi_{k+1} = (U_\mu^\times)^\ast \psi_k,
 \quad k\in \{1,2,\ldots\},
\end{equation}
or equivalently
\begin{equation}\label{Nr.3.15}
 \bigl( G_{\mu}^\ast (1), \psi_1 \bigr)_{L_\mu^2} > 0,
 \quad \psi_{k+1} (t) = t^k\cdot \psi_1 (t),
 \; t\in\mathbb T, \quad k\in \{1,2,\ldots\}.
\end{equation}
As in the paper \cite{Dub06} we introduce the following notion.

\begin{defn}\label{3.5}
 The constructed orthonormal basis
\begin{equation}\label{Nr.3.16}
 \varphi_0, \varphi_1, \varphi_2, \ldots ;\;  \psi_1, \psi_2, \ldots
\end{equation}
 in the space $L_\mu^2$ which satisfies the conditions \eqref{Nr.3.2} and \eqref{Nr.3.14} is called the
 cano\-ni\-cal orthonormal basis in $L_\mu^2$.
\end{defn}

Obviously, the canonical orthonormal basis \eqref{Nr.3.16} in $L_\mu^2$ is uniquely determined
by the conditions \eqref{Nr.3.2} and \eqref{Nr.3.14}. 
Here the sequence $(\varphi_k)_{k=0}^\infty$ is an orthonormal system of
polynomials (depending on $t^{-1}$). The orthonormal system $(\psi_k)_{k=1}^\infty$ is
built with the aid of the operator $U_\mu^\times$ from the function
$\psi_1$ $($see \eqref{Nr.3.13}$)$ in similar to the way the system
$(\varphi_k)_{k=0}^\infty$ was built from the function $\varphi_0$
$($see \eqref{Nr.3.1} and \eqref{Nr.3.2}$)$. The only difference is that the system
$\left( [(U_\mu^\times)^\ast ]^k \psi_1 \right)_{k=0}^\infty$ is
orthonormal, whereas in the general case the system $\left(
(U_\mu^\times)^k \varphi_0\right)_{k=0}^\infty$ is not
orthonormal. In this respect the sequence $(\psi_k)_{k=1}^\infty$
can be considered as a natural completion of the system of
orthonormal polynomials $(\varphi_k)_{k=0}^\infty$ to an orthonormal
basis in $L_\mu^2$.

It should be mentioned that the part $(\psi_k)_{k=1}^\infty$ of (\ref{Nr.3.16}) has a
clear interpretation in terms of prediction theory of stationary
sequences. A closer look at the papers Wiener/Masani \cite{WM1} and \cite{WM2}
shows that $(\psi_k)_{k=1}^\infty$ is the spectral
image of the sequence of so-called normalized innovations
corresponding to some stationa\-ry sequence which is naturally
associated with $L_\mu^2$. We hope to discuss this and
related matters in a separate paper dedicated to a
prediction-theoretical analysis of pseudocontinuability of Schur
functions.

\begin{rem}\label{3.6}
The orthonormal system
\begin{equation}\label{Nr.3.16b}
\varphi_1,\varphi_2,\ldots ;\;  \psi_1, \psi_2,\ldots
\end{equation}
is an orthonormal basis in the space $\gH_\mu$, which takes the orthogonal
decomposition of $\gH_\mu$ into account. We will call it
the canonical orthonormal basis in $\gH_\mu$. 
\end{rem}

In view of (\ref{Nr.3.15}), to obtain a description of the sequence $(\psi_k)_{k=1}^\infty$ 
it suffices to determine $\psi_1$. The above considerations lead us to the following result.

\begin{lem}\label{3.7}
 The function $\psi_1\in L_\mu^2$ is completely characterized by the following four conditions:
\begin{enumerate}
\item[(a)] 
 $\psi_1\perp \gH_{\mu,\gF}$. 
\item[(b)] 
 $\|\psi_1\| = 1$. 
\item[(c)] 
 $\bigl((V_{T_\mu^\ast})^k\psi_1\bigr)_{k=0}^\infty$ is an orthonormal basis
 in $\gH_{\mu,\gF}^\perp$. 
\item[(d)] 
 $\bigl( G_\mu^\ast (1), \psi_1 \bigr)_{L_\mu^2} > 0$.
\end{enumerate}
\end{lem}

\begin{rem}\label{3.8}
 In view of {\rm Theorem \ref{3.2}} condition {\rm (a)} in {\rm Lemma \ref{3.7}} is equivalent to the orthogonality
 conditions in $L_\mu^2$ which are expressed by
\[
\psi_1\perp \varphi_k,\quad k\in \{1,2, \ldots\},
\vspace{-2mm} \]
or equivalently \vspace{-2mm}
\[
\psi_1\perp e_{-k},\quad k\in \{1,2, \ldots\}.
\]
\end{rem}

Now we are going to prove another auxiliary result.
Here and in the following, for any function $f:\dT\to\dC$,
we simply write $f^{\ast}(t)$ instead of $(f(t))^{\ast}$.

\begin{lem}\label{3.9}
 Assume that $h_0\in L_\mu^2$ satisfies the following three conditions:
\begin{enumerate}
\item[($\tilde{\rm a}$)] 
 $h_0\perp \gH_{\mu,\gF}$.
\item[($\tilde{\rm b}$)] 
 $\|h_0\| = 1$.
\item[($\tilde{\rm c}$)] 
 $\bigl((V_{T_\mu^\ast})^k h_0\bigr)_{k=0}^\infty$ is an orthonormal system in $\gH_{\mu,\gF}^\perp$.
\end{enumerate}
 Suppose that $\mu$ has the Lebesgue decomposition \eqref{Nr.3.4} and let $D:\dD\to\dC$ be 
 the Szeg\H{o} function given by \eqref{Nr.3.5}. Denote by $E$ the Borel subset of $\mathbb T$
 which was introduced in {\rm Remark \ref{3.1}}. Then there is a function $I\in\cI(\dD)$ such that
 the function $h_0$ has the form
\begin{equation}\label{Nr.3.17}
 h_0 (t) = \left\{\begin{array}{cl}
  0, & t\in E, \\
 t \cdot \bigl(\underline{D}^{\ast}(t)\bigr)^{-1} \cdot \underline{I}(t) , 
     & t\in\mathbb T\setminus E.
\end{array}\right. 
\end{equation}
\end{lem}

\begin{proof}
In view of Remark \ref{3.8} from condition ($\tilde{\rm a}$) we infer
\begin{equation} \label{nr.f1}
\int_\mathbb T t^k h_0 (t) \,\mu (dt) 
 = \bigl(h_0, e_{-k}\bigr)_{L_{\mu}^2} 
 = 0, \quad k\in \{1,2, \ldots\}.
\end{equation}
Moreover, from ($\tilde{\rm c}$) and $(V_{T_\mu^\ast})^0 h_0=h_0$ we also have $h_0\in\gH_{\mu,\gF}^\perp$.
Thus, in view of \eqref{Nr.67} we get $h_0\in\gH_{\mu}$. This yields $h_0\perp \dC_{\dT}$.
Therefore, it follows that
\begin{equation} \label{nr.f2}
 \int_\mathbb T t^0 h_0 (t) \,\mu (dt) 
 = \bigl( h_0, e_{0}\bigr)_{L_{\mu}^2} = 0.
\end{equation}
By using the Riesz brothers' theorem (see, e.g., Hoffman \cite[Chapter~4]{H}) we obtain from \eqref{nr.f1} and \eqref{nr.f2}
we obtain that the complex measure $\overline{t}h_0 (t) \mu (dt)$ is absolutely continuous with
respect to $m$. According to the decomposition (\ref{Nr.3.4}) this implies
\begin{equation}\label{Nr.3.18}
h_0 (t) = 0,\quad t\in E_{\mu_s}.
\end{equation}
The function $h_0$ is determined $m$-a.e. on the set $\mathbb T\setminus E_{\mu_s}$. 
For this reason, taking into account Remark \ref{3.1} and \eqref{Nr.3.18}, we can assume that 
\begin{equation}\label{Nr.3.18b}
 h_0 (t) = 0,\quad t\in E.
\end{equation}
Because of (\ref{Nr.3.4}), (\ref{Nr.3.18}), and the condition ($\tilde{\rm c}$) for each 
$n\in \{1,2,\ldots\}$ it follows
\begin{equation*}
 \int_\mathbb T t^n | h_0 (t)|^2 |\underline{D} (t)|^2 \,m (dt) 
 = \int_\mathbb T t^n |h_0 (t)|^2 \,\mu (dt) 
 = \bigl( (V_{T_{\mu}^\ast})^n h_0, h_0 \bigr)_{L_{\mu}^2} = 0.
\end{equation*} 
Consequently, taking complex conjugates for each non--zero integer $n$, we get
\begin{equation*}
 \int_\mathbb T t^n | h_0 (t)|^2 |\underline{D} (t)|^2 \,m (dt) = 0.
\end{equation*}
Hence, there is a constant $c\in\mathbb C$ such that the identity
\begin{equation}\label{Nr.3.19}
 |h_0|^2 |\underline{D}|^2  = c
\end{equation}
holds $m$-a.e. on $\mathbb T$. Combining ($\tilde{\rm b}$) with
(\ref{Nr.3.4}), (\ref{Nr.3.18}), and (\ref{Nr.3.19}) we can conclude 
\[ 1 = \| h_0 \|^2 = \int_\mathbb T |h_0(t)|^2 \,\mu(dt) 
     = \int_\mathbb T |h_0(t)|^2 |\underline{D}(t)|^2 \,m(dt) = c. \]
Therefore, from (\ref{Nr.3.19}) it follows that
\begin{equation}\label{Nr.3.20}
 |h_0| = |\underline{D}|^{-1} 
\end{equation}
$m$-a.e. on $\mathbb T\setminus E$. From (\ref{Nr.3.20}) and Proposition~\ref{p2.6} we see that 
$e_{-1} h_0 |\underline{D}|^2\in L_m^2$. Using (\ref{Nr.3.4}), (\ref{Nr.3.18}), \eqref{nr.f1}, and \eqref{nr.f2}
we obtain
\begin{equation*}
 \int_\mathbb T t^n t^{-1} h_0 (t) |\underline{D} (t)|^2 \,m (dt) 
 = \int_\mathbb T t^{n-1} h_0 (t) \,\mu (dt) = 0,
 \quad n\in\{1,2,\ldots\}.
\end{equation*}
Thus, by setting $L_{m,+}^2 := \bigvee_{n=0}^\infty e_n$ we get
\begin{equation}\label{Nr.3.21}
 e_{-1}\cdot h_0\cdot |\underline{D}|^2\in (L_{m,+}^2)^\perp.
\end{equation}
In view of (\ref{Nr.3.20}), the identity
\begin{equation}\label{Nr.3.22}
 \bigl| e_{-1} \cdot h_0 \cdot |\underline{D}|^2 \bigr| = |\underline{D}|
\end{equation}
holds $m$-a.e. on $\mathbb T$. Since $D$ is an outer function in $H^2 (\mathbb D)$, it follows from
(\ref{Nr.3.21}), (\ref{Nr.3.22}) and the inner-outer factorization (see Proposition~\ref{p2.8}) that the function $e_{-1} h_0 |\underline{D}|^2$ admits the representation 
\begin{equation*}
 e_{-1} \cdot h_0 \cdot |\underline{D} |^2 = \underline{D}\cdot \underline{I} 
\end{equation*}
$m$-a.e. on $\dT$ for some $I\in\cI(\dD)$. Hence, the identity
\begin{equation*}
 e_{-1} \cdot h_0 \cdot \underline{D}^{\ast} = \underline{I} 
\end{equation*}
holds $m$-a.e. on $\dT$. Combining this with (\ref{Nr.3.18b}) we get the representation (\ref{Nr.3.17}).
\end{proof}

\begin{thm}\label{3.10}
 Let $D:\dD\to\dC$ be the Szeg\H{o} function given by \eqref{Nr.3.5}. Denote by $E$ the Borel subset 
 of $\mathbb T$ which was introduced in {\rm Remark \ref{3.1}}. Then the unit vector $\psi_1\in\tilde{\gL}_0$ 
 which is uniquely determined via \eqref{Nr.3.12} is given by
\begin{equation*}
 \psi_1 (t) = \left\{ \begin{array}{cl}
  0, & t\in E, \\
  t\cdot \bigl(\underline{D}^{\ast}(t)\bigr)^{-1},  & t\in\mathbb T\setminus E .
\end{array}\right.
\end{equation*}
\end{thm}

\begin{proof}
The conditions (a), (b), and (c) in Lemma \ref{3.7} lead in combination with Lemma~\ref{3.9} to the relation
\begin{equation}\label{Nr.3.23}
\psi_1 (t) = \left\{ \begin{array}{cl} 
  0, & t\in E, \\
  t\cdot \bigl(\underline{D}^{\ast}(t)\bigr)^{-1} \cdot \underline{I} (t), & t\in\mathbb T\setminus E .
\end{array}\right.
\end{equation}
Because of condition (c) in Lemma \ref{3.7}, the sequence
$\bigl((V_{T_\mu^\ast})^k\psi_1\bigr)_{k=0}^\infty$ is an orthonormal basis in $\gH_{\mu,\gF}^\perp$, 
where $(V_{T_\mu^\ast})^k \psi_1 = e_k\psi_1$,\, $k\in \{0,1,2,\ldots\}$. Now we are going
to prove that this implies that the function $I$ is constant with
unimodular value. We prove this  by contradiction. Assume that $I$
is not a unimodular constant function. Then Beur\-ling's Theorem
(see, e.g., Garnett \cite[Theorem 7.1 in Chapter~2]{G} implies that the system 
$\{ e_0\underline{I}, e_1\underline{I}, e_2\underline{I},\ldots\}$ is not closed in $L_{m,+}^2$. 
Thus, there exists an element $u\in L_{m,+}^2\setminus \{0\}$ which is orthogonal in $L_{m,+}^2$ to
the sequence $(e_n\underline{I})_{n=0}^\infty$. Now we show that the function
\begin{equation}\label{Nr.3.24}
g := \psi_1 \cdot \underline{I}^{\ast} \cdot u
\end{equation}
has the following properties
\begin{equation}\label{Nr.3.25}
g \in \gH_{\mu,\gF},
\end{equation}
\begin{equation}\label{Nr.3.26}
g\perp (V_{T_\mu^\ast})^n \psi_1, \quad n\in \{0,1,2,\ldots\}.
\end{equation}
Indeed, since the functions $u$ and $\underline{D}$ belong to $L_{m,+}^2$ from (\ref{Nr.3.24}), (\ref{Nr.3.23}), 
and the fact that $I\in\cI(\dD)$ for each $k\in \{0,1,2,\ldots \}$ we obtain
\begin{eqnarray*}
 \bigl(g, e_{-k}\bigr)_{L_\mu^2} 
 &=& \int_\mathbb T t^k \psi_1 (t) \underline{I}^{\ast}(t) u (t) \,\mu (dt) \\
 &=& \int_\mathbb T t^k t \bigl(\underline{D}^{\ast} (t)\bigr)^{-1} \underline{I} (t) \underline{I}^{\ast}(t)
     u (t) \,\mu (dt) \\
 &=& \int_\mathbb T t^{k+1} \bigl(\underline{D}^{\ast} (t)\bigr)^{-1} u (t) \,\mu (dt) \\
 &=& \int_\mathbb T t^{k+1} u (t) \underline{D} (t) \,m(dt) \,=\, 0.
\end{eqnarray*}
This implies (\ref{Nr.3.25}). Moreover, taking into account (\ref{Nr.3.24}), (\ref{Nr.3.23}), and the
choice of $u$ for each $k\in \{0,1,2,\ldots\}$ one can conclude 
\begin{eqnarray*}
 \bigl(g, (V_{T_\mu^\ast})^k \psi_1\bigr)_{L_\mu^2} 
 &=& \bigl(g, e_k\psi_1\bigr)_{L_\mu^2} 
 \;=\;\int_\mathbb T g(t) \bigl(t^k\psi_1 (t)\bigr)^{\ast} \,\mu (dt) \\
 &=& \int_\mathbb T g(t)  \bigl(t^k\psi_1 (t)\bigr)^{\ast} |\underline{D} (t)|^2 \,m (dt) \\
 &=& \int_\mathbb T u(t)  \bigl(\underline{I}(t) t^k\bigr)^{\ast} |\psi_1(t)|^2 |\underline{D}(t)|^2 \,m(dt) \\ 
 &=& \int_\mathbb T u(t)  \bigl(\underline{I}(t) t^k\bigr)^{\ast}
 \left| t \bigl(\underline{D}^{\ast} (t)\bigr)^{-1} \underline{I} (t)\right |^2 |\underline{D} (t)|^2 \,m(dt) \\
 &=& \int_\mathbb T u(t) \bigl(t^k \underline{I} (t)\bigr)^{\ast} \,m(dt) 
 \;=\; \bigl(u, e_k\underline{I}\bigr)_{L_m^2} \;=\;  0.
\end{eqnarray*}
Thus, (\ref{Nr.3.26}) is proved. However, condition (c) in Lemma \ref{3.7} shows that from (\ref{Nr.3.25})
and (\ref{Nr.3.26}) it follows that $g=0$. Combining this with (\ref{Nr.3.23}) and (\ref{Nr.3.24}) we get
\[ e_1\cdot (\underline{D}^{\ast})^{-1} \cdot  u = 0 \]
$m$-a.e. on $\mathbb T\setminus E$. Therefore, $u$ vanishes $m$-a.e. on $\mathbb T\setminus E$. 
Because $m(E) = 0$, this implies that $u$ vanishes $m$-a.e. on $\mathbb T$. 
This contradicts the assumption that $u$ belongs to
$L_{m,+}^2\setminus \{0\}$. Hence, the function $I$ is constant. Since $I$ is an inner function, there is an 
$\alpha\in\mathbb R$ such that $I$ is the constant function with value ${\rm exp} \{i\alpha\}$. 
Thus, formula (\ref{Nr.3.23}) can be rewritten in the form
\begin{equation}\label{Nr.3.27}
 \psi_1 (t) = \left \{\begin{array}{cl}
  0, & t\in E, \\
  \exp \{i\alpha\} \cdot t \cdot \bigl(\underline{D}^{\ast}(t)\bigr)^{-1}, 
     & t\in \mathbb T\setminus E .
\end{array}\right.
\end{equation}
Denote by $P_{\mathbb C_\mathbb T}$ the orthoprojector in $L_\mu^2$
onto the closed subspace $\mathbb C_\mathbb T$. Using \eqref{Nr.3.5},
(\ref{Nr.3.27}), and (\ref{Nr.3.12}) we obtain
\begin{eqnarray*}
0 \;<\; \bigl(G_\mu^\ast (1), \psi_1\bigr)_{L_\mu^2} 
  &\!=&\!\bigl(1,G_\mu \psi_1\bigr)_{\mathbb C} \;=\; G_\mu\psi_1 
 \;=\;\bigl(0,\;1\bigr)
      \left(\begin{matrix} T_\mu & F_\mu\\ G_\mu & S_\mu\end{matrix}\right) 
      \left(\begin{matrix} \psi_1 \cr 0\end{matrix}\right) \cr
  &\!=&\!P_{\mathbb C_{\mathbb T}} U_\mu^\times \psi_1 
 \;=\;\int_\mathbb T t^{\ast} \psi_1 (t) \,\mu (dt) 
 \;=\;\int_\mathbb T t^{\ast} \psi_1 (t) |\underline{D} (t)|^2 \,m (dt) \cr 
  &\!=&\!\int_\mathbb T t^{\ast} \left( \exp \{i\alpha \} t \bigl(\underline{D}^{\ast} (t)\bigr)^{-1} \right)
      |\underline{D} (t)|^2 \,m (dt) \cr
  &\!=&\!\exp \{ i\alpha\} \cdot \int_\mathbb T \underline{D} (t) \,m (dt)
 \;=\;\exp \{ i\alpha\} \cdot D (0) \cr
  &\!=&\!\exp \{ i\alpha\} \cdot \exp \left\{ \frac{1}{2} \int_\mathbb T {\rm ln}\,[w(t)] \,m (dt) \right\} .
\end{eqnarray*}
Consequently, it follows $\exp \{ i\alpha\}=1$.
\end{proof}

\begin{cor}\label{c3.9}
For each $k\in \{1,2,\ldots\}$, the function $\psi_k$ of the canonical orthonormal basis {\rm (\ref{Nr.3.16b})} 
in the space $\gH_\mu$ is given by
\begin{equation}\label{Nr.3.28}
 \psi_k (t) = 
 \left\{ \begin{array}{cl} 0, & t\in E, \cr 
  t^k \cdot \bigl(\underline{D}^{\ast} (t)\bigr)^{-1}, & t\in\mathbb T\setminus E . 
\end{array}\right. 
\end{equation}
\end{cor}

\begin{cor}\label{c3.10}
The canonical orthonormal basis \eqref{Nr.3.16} in the space $L_\mu^2$ consists of the system 
$(\varphi_k)_{k=0}^\infty$ of orthonormal polynomials $($depending on $t^{-1})$ given by 
\eqref{Nr.3.2} and the orthonormal system $(\psi_k)_{k=1}^\infty$ given by \eqref{Nr.3.28}.
\end{cor}

\begin{rem}\label{3.11}
If $\mu$ coincides with the Lebesgue-Borel measure $m$ on $\mathbb T$, then the
Szeg\H{o} condition for $\mu$ is satisfied and $D$ is the constant function with value $1$. Thus,
\[
\varphi_k (t) = t^{-k},\quad k\in \{0,1,2,\ldots\},
\vspace{-2mm} \]
and \vspace{-1mm}
\[
\psi_k (t) = t^k,\quad k\in \{1,2,\ldots\}.
\]
In other words, the canonical orthonormal basis \eqref{Nr.3.16} in that case is given by 
\[ e_0, e_{-1},e_{-2},\ldots ;\; e_1, e_2,\ldots, \]
where for each integer $n$ the function $e_n :\mathbb T\to\mathbb C$ is defined as in \eqref{Nr.EK}.
\end{rem}

\begin{cor}\label{3.12}
The space $\gH_{\mu,\gF}^\perp$ consists of all functions $f$ of the form
\begin{equation*}
 f(t) = \left\{ \begin{array}{cl}
  0, & t\in E, \cr
  f_0 (t)\cdot\bigl(\underline{D}^{\ast} (t)\bigr)^{-1} , 
     & t\in\mathbb T\setminus E ,
\end{array}\right. 
\end{equation*}
where $f_0$ is some function belonging to $L_{m,+}^2$. Thereby,
\[
\|f\|_{L_\mu^2} = \|f_0\|_{L_m^2},
\]
i.e. the mapping $f\mapsto f_0$ establishes a metric isomorphism
between $\gH_{\mu,\gF}^\perp$ and $L_{m,+}^2$.
\end{cor}

Now we consider the case that the measure $\mu\in\cM_+^1 (\mathbb T)$
is generated by the pseudocontinuable non--inner function $\Theta$.

\renewcommand{\labelenumi}{(\alph{enumi})}
\renewcommand{\theenumi}{(\alph{enumi})}

\begin{prop}      \label{p6.14a}
  Let $\Theta \in \cS\Pi(\mathbb D)\setminus\cI(\mathbb D)$
  and $\mu\in\cM_+^1 (\mathbb T)$ be the measure associated with
  $\Theta$ according to Section \ref{s1}. Then:
  \begin{enumerate}
    \item      \label{p6.14a-Ta}
      The measure $\mu$ satisfies the Szeg\H{o} condition
      and the Szeg\H{o} function $D$ associated with $\mu$
      belongs to $\Pi(\mathbb D)$.
      
    \item      \label{p6.14a-Tb}
      Denote by $D^{\#}$ the pseudocontinuation of $D$.
      Let $k \in {\mathbb N}$. For all $\zeta \in {\mathbb D}$
      such that $D^{\#}$ is holomorphic at $\displaystyle \frac{1}{ \zeta^* }$
      and satisfies
      $\displaystyle D^{\#} \left( \frac{1}{ \zeta^* } \right) \neq 0$,
      we define
      \begin{equation*}
        \Psi_k(\zeta) 
          := \zeta^k
              \cdot
              \frac{1}
                  {
                  \left[
                    D^{\#} \left( \frac{1}{ \zeta^* } \right)
                  \right]^*
                  }
      \end{equation*}
      Then $\Psi_k(\zeta) \in \cN\cM({\mathbb D})$ and its boundary
      values $\underline{ \Psi_k }$ satisfy
      $1_{{\mathbb T} \setminus E} \underline{ \Psi_k } = \psi_k$
      $\mu$--a.e. on ${\mathbb T}$, where $\psi_k$ belongs to the canonical
      orthonormal basis in $L_{\mu}^2$ which was introduced in
      Definition \ref{3.5} and where $E$ stands for the Borel
      subset of ${\mathbb T}$, which was introduced in Remark \ref{3.1}.
  
    \item      \label{p6.14a-Tc}
      Let $F \in H^2({\mathbb D})$ and let
      \begin{equation*}
        G :=  \frac{F}
                {
                  \left(
                    D^{\#} 
                  \right)^{\wedge}
                }  .
      \end{equation*}
      Then $G \in \cN\cM({\mathbb D})$ and its boundary values
      $\underline{G}$ satisfy
      \begin{equation*}
        1_{{\mathbb T} \setminus E} \underline{G} \in \gH_{\mu,\gF}^\perp
      \end{equation*}
      and
      \begin{equation*}
        \left\lVert
          F
        \right\rVert_{H^2({\mathbb D})}
        =
        \left\lVert
          1_{{\mathbb T} \setminus E} \underline{G}
        \right\rVert_{L_{\mu}^2}  .
      \end{equation*}
      The mapping $F \longmapsto 1_{{\mathbb T} \setminus E} \underline{G}$
      is a metric isomorphism between $H^2({\mathbb D})$
      and $\gH_{\mu,\gF}^\perp$.
  \end{enumerate}
\end{prop}

\begin{proof}
  \ref{p6.14a-Ta}
    This follows from Theorem \ref{t3.9}.
  
  \ref{p6.14a-Tb}
    In view of Remark \ref{r2.10} we have $\Psi_k \in \cN\cM({\mathbb D})$
    and
    \begin{equation}      \label{p6.14a-1}
      \underline{ \Psi_k }(t) = \frac{t^k}{ \underline{D}^*(t) }
      \qquad  \qquad
      \mbox{for $m$--a.e. $t \in {\mathbb T}$.}
    \end{equation}
    Thus, in view of $m(E) = 0$, the combination of \eqref{p6.14a-1},
    Theorem \ref{3.10} and \eqref{Nr.3.15} shows that
    $1_{{\mathbb T} \setminus E} \underline{ \Psi_k } = \psi_k$
    $\mu$--a.e. on ${\mathbb T}$.
    
  \ref{p6.14a-Tc}
    Using Corollary \ref{3.12}, the assertion of \ref{p6.14a-Tc}
    can be obtained analogously \mbox{to \ref{p6.14a-Tb}}.
\end{proof}

\begin{rem}\label{3.13}
In \cite{Dub06} a detailed description of the matrix representations of the o\-pe\-rators
$T_\mu$, $F_\mu$, $G_\mu$, and $S_\mu$ $($and, consequently, of the operator $U_\mu)$ with respect
to the canonical basis \eqref{Nr.3.16} was given $($see 2.8 Comment, Chapter~A, and Theorems 2.13, 2.15,
and 2.17 there$)$. These matrix representations are expressed in terms of Schur parameters of the
characteristic function $\Theta_{\bigtriangleup_\mu}$ of the unitary colligation which is associated
with the probability measure $\mu$. We note that in Chap\-ter~4 of the monograph \cite{Sim05} by B. Simon
some historical remarks concerning these matrix representations are presented. In particular, the
representation, abbreviated by $GGT$, is used in \cite{Sim05}. Here $GGT$ stands for Geronimus, Gragg,
and Teplyaev.
\end{rem}


\section{Some criterion of pseudocontinuability of non-inner Schur functions 
in terms of orthogonal polynomials on the unit circle}
\label{s7}


It is well known (see, e.g., Brodskii \cite{B}) that one can consider simultaneously together with the simple unitary
colligation (\ref{Nr.2.2}) the adjoint unitary colligation
\begin{equation}\label{Nr.3.29}
 \tilde{\bigtriangleup}_\mu := 
 (\gH_\mu, \mathbb C, \mathbb C;T_\mu^\ast, G_\mu^\ast, F_\mu^\ast, S_\mu^\ast)
\end{equation}
which is also simple. For
$z\in\mathbb D$, its characteristic function $\Theta_{\tilde{\bigtriangleup}_\mu}$ is  given by
\begin{equation*}
 \Theta_{\tilde{\bigtriangleup}_\mu} (z) = \Theta_{\bigtriangleup_\mu}^{\ast} (z^{\ast}).
\end{equation*}
We note that the unitary colligation (\ref{Nr.3.29}) is associated with the operator $(U_\mu^\times)^\ast$.
It can be easily checked that the action of $(U_\mu^\times)^\ast$ is given for each
$ f\in L_\mu^2$ by
\[ [(U_\mu^\times)^\ast f] (t) = t\cdot f(t),\quad t\in\mathbb T . \]
If we replace the operator $U_\mu^\times $ by $(U_\mu^\times)^\ast$ in the preceding considerations,
which have lead to the canonical orthonormal basis (\ref{Nr.3.16}), we obtain an orthonormal basis
of the space $L_\mu^2$ which consists of two sequences
\begin{equation}
  \label{Nr.ONB}
  (\tilde{\varphi}_j)_{j=0}^\infty 
\quad\mbox{and}\quad
  (\tilde{\psi}_j)_{j=1}^\infty
\end{equation}
of functions. From our treatments above it follows that the orthonormal basis (\ref{Nr.ONB}) is uniquely
determined by the following conditions:
\begin{enumerate}
\item[(a)]
The sequence $(\tilde{\varphi}_j)_{k=0}^\infty$ arises from the result of the Gram-Schmidt orthogonalization
procedure of the sequence $\left( [(U_\mu^\times)^\ast]^n {\bf 1}\right)_{n=0}^\infty$ while
taking into account the normalization conditions
\begin{equation*}
 \bigl( [(U_\mu^\times)^\ast]^n {\bf 1}, \tilde{\varphi}_n \bigr)_{L_\mu^2} > 0,
\quad n\in \{0,1,2,\ldots \}.
\end{equation*}
\item[(b)]
The relations
\[ \bigl(F_{\mu}(1), \tilde{\psi}_1\bigr)_{L_\mu^2} > 0
\quad\mbox{and}\quad
   \tilde{\psi}_{k+1} = U_\mu^\times \tilde{\psi}_k,\quad k\in \{1,2,\ldots\}, \]
hold. 
\end{enumerate}
It can be easily checked that
\[ \tilde{\varphi}_{k} = \varphi_k^{\ast}, 
\quad k\in \{0,1,2,\ldots\}, 
\vspace{-2mm} \]
and \vspace{-2mm}
\[ \tilde{\psi}_{k} = \psi_k^{\ast},
\quad k\in \{1,2,\ldots\}.
\]

As in the paper \cite{Dub06}, we introduce the following notion.

\begin{defn}\label{3.14}
The orthogonal basis
\begin{equation}\label{Nr.3.31}
 \varphi_0^{\ast}, \varphi_1^{\ast}, \varphi_2^{\ast},\ldots ; \psi_1^{\ast}, \psi_2^{\ast}, \ldots
\end{equation}
is called the conjugate canonical orthonormal basis with respect to the canonical orthonormal basis \eqref{Nr.3.16}.
\end{defn}

Similar to (\ref{Nr.3.8}), the identity
\begin{equation}\label{Nr.3.38}
 \bigvee_{k=0}^{n} [(U_\mu^\times)^\ast]^k {\bf 1} 
 = \left( \bigvee_{k=0}^{n-1} (T_\mu^\ast)^k G_\mu^\ast (1) \right) \oplus \mathbb C_\mathbb T
\end{equation}
can be verified. Thus,
 \begin{equation}\label{Nr.3.39}
\gH_{\mu,\gF} = \bigvee_{k=1}^\infty \varphi_k,\quad \gH_{\mu,\gG}
= \bigvee_{k=1}^\infty \varphi_k^{\ast},
\end{equation}
 \begin{equation}\label{Nr.3.40}
\gH_{\mu,\gF}^\perp = \bigvee_{k=1}^\infty \psi_k,\quad
\gH_{\mu,\gG}^\perp = \bigvee_{k=1}^\infty \psi_k^{\ast}.
\end{equation}

In \cite[Chapter 3]{Dub06} the unitary operator $U$ was introduced  which
maps the elements of the canonical basis (\ref{Nr.3.16}) onto the corresponding elements of the
conjugate canonical basis (\ref{Nr.3.31}). More precisely,
\begin{equation*}
 U \varphi_n = \varphi_n^{\ast},\quad n\in \{0,1,2,\ldots\},
\qquad \mbox{and} \qquad 
 U \psi_n = \psi_n^{\ast},\quad n\in \{1,2,\ldots\}.
\end{equation*}
The operator $U$ is related to the conjugation operator in $L_\mu^2$. Namely, if
$f\in L_\mu^2$~and
 \begin{equation*}
f =\sum_{k=0}^\infty \alpha_k \varphi_k + \sum_{k=1}^\infty
\beta_k\psi_k,
\end{equation*}
then
\begin{equation*}
 f^{\ast} = \sum_{k=0}^\infty \alpha_k^{\ast} \varphi_k^{\ast} + \sum_{k=1}^\infty \beta_k^{\ast} \psi_k^{\ast} 
          = \sum_{k=0}^\infty \alpha_k^{\ast} U\varphi_k + \sum_{k=1}^\infty \beta_k^{\ast} U\psi_k.
\end{equation*}
In \cite[Theorem~3.6]{Dub06} the matrix representation of the blocks of the operator $U$ with respect
to the canonical basis (\ref{Nr.3.16}) was expressed in terms of Schur parameters of the
characteristic function $\Theta_{\mu_\bigtriangleup}$. A finer analysis of this matrix representation
(see \cite[Chapters 4 and 5]{Dub06}) yields effective criteria for the pseudocontinuability of a non-inner
function belonging to the Schur class $\cS (\mathbb D)$. These criteria are formulated in terms of
Schur parameters.

Let $\Theta\in\cS (\mathbb D).$ The relations (\ref{Nr.0.1f}) and (\ref{Nr.0.3}) allow us to associate a measure $\mu\in \cM_+^1 (\mathbb T)$ with the function $\Theta$. Furthermore, we associate the simple unitary colligation $\bigtriangleup_\mu$ given by (\ref{Nr.2.2}) with the same measure $\mu$ . 
Then Theorem \ref{2.1} implies $\Theta_{\bigtriangleup_\mu} =\Theta$. The unitary colligation 
$\bigtriangleup_\mu$ suggests that we consider the canonical basis (\ref{Nr.3.16}) and 
the conjugate canonical basis (\ref{Nr.3.31}) in $L_\mu^2$. The equations (\ref{Nr.3.39}) and (\ref{Nr.3.40}) 
allow us to reformulate Theorem \ref{t4.2} into the following criterion of pseudocontinuability of 
non-inner Schur functions in terms of the canonical orthonormal basis
of $L_\mu^2$ introduced in Definition \ref{3.5}.

\begin{thm}\label{4.3}
 Let $\Theta\in\cS (\mathbb D)\setminus\cI(\mathbb D)$ and let
 $\mu\in\cM_+^1 (\mathbb T)$ be the measure associated with
 $\Theta$ according to Section \ref{s1}. Let
 $\bigtriangleup_\mu$ be the simple unitary colligation 
\eqref{Nr.2.2} which satisifies $\Theta_{\bigtriangleup_\mu} =\Theta$. 
 Furthermore, let the cano\-ni\-cal orthogonal basis in $L_\mu^2$ be
 given by \eqref{Nr.3.16}. Then $\Theta$ admits a pesudocontinuation if
 and only if
  \begin{equation}\label{Nr.4.2}
  \bigvee_{n=1}^\infty \varphi_n^{\ast}\cap \bigvee_{n=1}^\infty\psi_n \ne \{0\}
  \vspace{-1mm}
  \end{equation}
  (where the symbol ,,$\bigvee$`` is used in the context of $L_\mu^2$).
\end{thm}



\section{A direct connection between the criteria of pseudocontinuability given by Theorems \ref{t3.9} and \ref{t4.2}}
\label{s8}


The combination of Proposition \ref{p2.14} and Theorem \ref{4.3} shows that in order to establish a 
direct connection between the criteria of pseudocontinuability of functions belonging to 
$\cS (\mathbb D)\setminus\cI(\mathbb D)$ which are contained in Theorems \ref{t3.9} and \ref{t4.2}
it is sufficient to verify the following statement.

\begin{thm}\label{4.5}
 Let $\Theta\in\cS(\mathbb D)\setminus\cI(\mathbb D)$ and let $\mu\in\cM_{+}^{1}(\dT)$ be the measure
 associated with $\Theta$ according to {\rm Section~\ref{s1}}. Then condition \eqref{Nr.4.2} in {\rm Theorem~\ref{4.3}} is fulfilled if and only if the measure $\mu$ satisfies 
 the Szeg\H{o} condition and there exist functions $I_1,I_2\in\cI(\mathbb D)$ such that the identity
 \eqref{Nr.p1} holds $m$-a.e. on $\dT$.
\end{thm}

\begin{proof}
Let condition (\ref{Nr.4.2}) be satisfied. Then the system $(\varphi_n)_{n=0}^\infty$ is not complete
in $L_\mu^2$. Thus, as mentioned above, condition (\ref{Nr.3.4a}) is fulfilled. From the definition of 
the subspaces $\gH_{\mu,\gF}$ and $\gH_{\mu,\gG}$ (see (\ref{Nr.d1.5c}) and Section~\ref{s5}) it follows that the
subspaces $\gH_{\mu,\gG}$ and $\gH_{\mu,\gF}^\perp$ are invariant
with respect to the operator $T_\mu^\ast$. Moreover, since the
unitary colligation $\bigtriangleup_\mu$ is simple we get
\begin{equation*}
{\rm Rstr.}_{\gH_{\mu,\gF}^\perp} T_\mu^\ast = V_{T_\mu^\ast}.
\end{equation*}
Taking into account (\ref{Nr.3.39}), (\ref{Nr.3.40}), and (\ref{Nr.4.2}) it follows
\begin{equation*}
\gN_{\gG\gF} := \gH_{\mu,\gG} \cap \gH_{\mu,\gF}^\perp =
\bigvee_{n=1}^\infty \varphi_n^{\ast} \cap
\bigvee_{n=1}^\infty \psi_n\ne \{0\}.
\end{equation*}
Hence, the space $\gN_{\gG\gF}$ is also invariant with respect
to $T_\mu^\ast$ and the restriction of $T_\mu^\ast$ is a
unilateral shift of multiplicity $1$. Let $h_0$ be a basis function
of the gene\-ra\-ting wandering subspace of the unilateral shift 
${\rm Rstr.}_{\gN_{\gG\gF}} T_\mu^\ast$. Then the function $h_0$
satisfies the conditions in Lemma \ref{3.9}. Thus, there is a
function $\hat{I}_1\in\cI(\dD)$ such~that 
\begin{equation}\label{Nr.4.5}
 h_0 (t) = \left\{\begin{array}{cl}
     0 , & t\in E, \\
 t\cdot \bigl(\underline{D}^{\ast} (t)\bigr)^{-1} \cdot \underline{\hat{I}_1} (t) , 
         & t\in\mathbb T\setminus E .
\end{array}\right.
\end{equation}
On the other hand, in view of $h_0\in \bigvee_{n=1}^\infty \varphi_n^{\ast}$ there exists a
sequence $(P_n)_{n=1}^\infty$ of polynomials such that the limit relation
\begin{equation*}
\lim_{n\to \infty} \int_\mathbb T |h_0(t) - P_n(t)|^2 \,\mu(dt) = 0
\end{equation*}
holds. This implies
\begin{equation*}
\lim_{n\to\infty} \left\{ \int_\mathbb T |h_0 (t)\underline{D} (t) - P_n (t)\underline{D} (t)|^2 \,m(dt)
+ \int_\mathbb T |P_n (t)|^2 \,\mu_s (dt)\right\} = 0.
\end{equation*}
Combining this with (\ref{Nr.4.5}) we obtain
\begin{equation*}
\lim_{n\to\infty}  \int_\mathbb T
 \left | t \cdot \underline{\hat{I}_1}(t) \cdot \underline{D}(t) \cdot \bigl(\underline{D}^{\ast} (t)\bigr)^{-1} 
        - P_n(t) \underline{D} (t) \right|^2 m (dt) = 0.
\end{equation*}
Therefore, the unimodular function $\hat{I} :\mathbb T\to\mathbb C$ which is defined by
\begin{equation*}
\hat{I}(t)  := t \cdot \underline{\hat{I}_1} (t) \cdot \underline{D}(t) \cdot \bigl(\underline{D}^{\ast} (t)\bigr)^{-1} 
\end{equation*}
is the $L_m^2$-limit of the sequence $\bigl(({\rm Rstr.}_\mathbb T P_n)\cdot \underline{D}\bigr)_{n=1}^\infty$.
Thus, since the sequence $\bigl(({\rm Rstr.}_\mathbb D P_n)\cdot D\bigr)_{n=1}^\infty$ belongs to the Hardy space
$H^2 (\mathbb D)$ and since $\hat{I}$ is a unimodular function, we see that there is a function $I_2\in\cI(\dD)$ 
such that its boundary function $\underline{I}_2$ coincides $m$-a.e. on $\dT$ with $\hat{I}$. Hence, the identity
\begin{equation*}
 \underline{I_2} (t)  = 
 t \cdot \underline{\hat{I}_1} (t) \cdot  \underline{D}(t) \cdot \bigl(\underline{D}^{\ast} (t)\bigr)^{-1} 
\end{equation*}
holds $m$-a.e. on $\dT$. This implies that 
\begin{equation*}
 \underline{D}(t) \cdot \bigl(\underline{D}^{\ast} (t)\bigr)^{-1}  =
 \underline{I_2} (t) \cdot t^{-1} \cdot \bigl(\underline{\hat{I}_1} (t)\bigr)^{-1}  
\end{equation*}
holds $m$-a.e. on $\dT$.
Therefore, if we define the function $I_1 :\mathbb D\to\mathbb C$ by 
\[ I_1 (\zeta) := \zeta\cdot \hat{I}_1 (\zeta), \]
then $I_1$ is a function belonging to $\cI(\mathbb D)$ such that 
\eqref{Nr.p1} holds $m$-a.e. on $\dT$.

Conversely, we assume now that the conditions \eqref{Nr.3.4a} and \eqref{Nr.p1} are satisfied. Let $E$ be the set 
introduced in Remark \ref{3.1}. In view of $m(E) = 0$ we can assume that the 
boundary functions $\underline{I_1}$ and $\underline{I_2}$ of the inner functions 
$I_1$ and $I_2$ vanish on $E$. Thus, the functions $\underline{I_1}$ and $\underline{I_2}$ 
vanish $\mu_s$-a.e. on $\dT$. From \eqref{Nr.p1} it follows that
\begin{equation}\label{Nr.4.7}
 t\cdot \underline{I_1} (t) \cdot \bigl(\underline{D}^{\ast} (t)\bigr)^{-1} =
 t\cdot \underline{I_2} (t) \cdot \bigl(\underline{D} (t)\bigr)^{-1} 
\end{equation}
holds $m$-a.e. on $\dT$.
We are going to prove that the function $h_0$ defined by
\begin{equation}\label{Nr.4.8}
 h_0 (t) := \left\{\begin{array}{cl}
 0 , & t\in E, \\
 t\cdot \underline{I_1} (t) \cdot \bigl(\underline{D}^{\ast} (t)\bigr)^{-1} ,
    & t\in\mathbb T\setminus E,
\end{array}\right.
\end{equation}
belongs to the intersection of the sets described on the left-hand side of formula (\ref{Nr.4.2}).
Indeed, in view of Corollary \ref{3.12} and (\ref{Nr.3.40}) we infer
\begin{equation}\label{Nr.4.9}
h_0\in \bigvee_{n=1}^\infty \psi_n.
\end{equation}
On the other hand, using (\ref{Nr.4.8}), (\ref{Nr.4.7}), (\ref{Nr.3.28}), and $m(E) = 0$ for each
$n\in\{1,2,\ldots\}$ one can conclude 
\begin{eqnarray*}
 \bigl(h_0,\psi_n^{\ast}\bigr)_{L_\mu^2} 
  &=& \int_\mathbb T h_0(t) \psi_n(t) \,\mu(dt)
 \;=\;\int_{\mathbb T\setminus E} h_0(t)\psi_n(t) \,d\mu \\
  &=& \int_{\mathbb T\setminus E } 
     t \underline{I_1} (t) \bigl(\underline{D}^{\ast} (t)\bigr)^{-1} \psi_n (t) \,\mu (dt) \\
  &=& \int_{\mathbb T\setminus E } 
     t \underline{I_2} (t) \bigl(\underline{D} (t)\bigr)^{-1} \psi_n (t) \,\mu (dt) \\
  &=& \int_{\mathbb T\setminus E} t \underline{I_2} (t) \bigl(\underline{D} (t)\bigr)^{-1}
        t^n \bigl(\underline{D}^{\ast} (t)\bigr)^{-1} \,\mu (dt) \\
  &=& \int_{\mathbb T\setminus E} t^{n+1} \underline{I_2} (t) |\underline{D} (t)|^{-2} \,\mu (dt) \\
  &=& \int_{\mathbb T\setminus E} t^{n+1} \underline{I_2} (t) |\underline{D} (t)|^{-2} |\underline{D} (t)|^2 \,m (dt) \\
  &=& \int_{\mathbb T\setminus E} t^{n+1}  \underline{I_2} (t) \,m (dt)
 \;=\;\int_{\mathbb T} t^{n+1}  \underline{I_2} (t) \,m (dt) \;=\; 0.
\end{eqnarray*}
Thus, we get
\begin{equation}\label{Nr.4.10}
 h_0\in \left(\bigvee_{n=1}^\infty \psi_n^{\ast}\right)^\perp .
\end{equation}
In view of $I_1 D\in H^2 (\mathbb D)$ we have
\[ \int_\mathbb T t \underline{I_1} (t) \underline{D} (t) \,m(dt) = 0. \]
Consequently, by using (\ref{Nr.4.11}), (\ref{Nr.4.8}), and $m(E)=0$ we obtain
\begin{eqnarray}\label{Nr.4.13}
 \bigl(h_0,\varphi_0^{\ast}\bigr)_{L_\mu^2} 
  &\!=&\! \int_\mathbb T h_0(t) \varphi_0(t) \,\mu(dt)
 \;=\;\int_\mathbb T h_0(t) \,\mu(dt) 
 \;=\;\int_{\mathbb T\setminus E } h_0(t) \,\mu(dt) \cr
  &\!=&\!\int_{\mathbb T\setminus E} t \underline{I_1} (t) \bigl(\underline{D}^{\ast} (t)\bigr)^{-1} \,\mu (dt) 
 \;=\;\int_{\mathbb T\setminus E} t \underline{I_1} (t) \bigl(\underline{D}^{\ast} (t)\bigr)^{-1}
       |\underline{D} (t)|^2 \,m (dt) \cr
  &\!=&\!\int_{\mathbb T\setminus E} t \underline{I_1} (t) \underline{D} (t) \,m (dt) 
 \;=\;\int_\mathbb T t \underline{I_1} (t) \underline{D} (t) \,m (dt)
 \;=\;0.\qquad
\end{eqnarray}
Since
\begin{equation*}
 \varphi_0^{\ast}, \varphi_1^{\ast}, \varphi_2^{\ast}, \ldots ; \psi_1^{\ast}, \psi_2^{\ast}, \ldots
\end{equation*}
is an orthonormal basis of $L_\mu^2$ from (\ref{Nr.4.10}) and (\ref{Nr.4.13}) it follows that
\begin{equation*}
 h_0\in \bigvee_{n=1}^\infty \varphi_n^{\ast} .
\end{equation*}
Combining this with (\ref{Nr.4.9}) we see that
\begin{equation}\label{Nr.4.14}
 h_0\in \bigvee_{n=1}^\infty \varphi_n^{\ast} \cap \bigvee_{n=1}^\infty \psi_n.
\end{equation}
By using $I_1\in\cI(\dD)$, $m(E) = 0$, and (\ref{Nr.4.8}) we obtain
\begin{eqnarray*}
\|h_0\|_{L_\mu^2}^2 
 &=& 
 \int_\mathbb T |h_0(t)|^2 \,\mu(dt) \;=\; \int_{\mathbb T\setminus E } |h_0(t)|^2 \,\mu(dt) \\
 &=&
 \int_{\mathbb T\setminus E} \left|t \underline{I_1} (t) \bigl(\underline{D}^{\ast} (t)\bigr)^{-1} \right|^2 \mu (dt)
 \;=\;
 \int_{\mathbb T\setminus E} |\underline{D} (t)|^{-2} \,\mu (dt) \\
 &=&
 \int_{\mathbb T\setminus E} |\underline{D} (t)|^{-2} |\underline{D} (t)|^2 \,m(dt) 
 \;=\; m (\mathbb T \setminus E) \;=\; m (\mathbb T) \;=\; 1 .
\end{eqnarray*}
Therefore, from (\ref{Nr.4.14}) we see
\begin{equation*}
\bigvee_{n=1}^\infty \varphi_n^{\ast} \cap
\bigvee_{n=1}^\infty \psi_n\ne \{0\}.
\end{equation*}
Thus, (\ref{Nr.4.2}) is fulfilled. 
\end{proof}


We now look to cast some light on an interesting aspect of pseudocontinuability of Schur functions by considering
our main result under the light of the following well-known Douglas-Rudin Theorem
(see, e.g., Garnett \cite[Theorem 2.1 in Chapter~V]{G}).

\begin{thm}\label{4.7}
Denote by $U$ a unimodular function belonging to the space $L^\infty(\mathbb T)$. 
Then for every $\varepsilon>0$ there exist functions $I_1,I_2\in\cI(\dD)$ satisfying
\begin{equation*}
\left\| U - \underline{I_2} \cdot \underline{I_1}^{-1}\right \|_{L^\infty (\mathbb T)} <\varepsilon.
\end{equation*}
\end{thm}

Thus, the combination of Theorem~\ref{4.3} with Theorem~\ref{4.5} shows that the phenomenon
of pseudocontinuability of a non-inner function $\Theta\in\cS (\mathbb D)$ is related to the fact
that for the unimodular function
\[
U := \underline{D} \cdot (\underline{D}^{\ast})^{-1}
\]
an extremal situation can be met in Theorem \ref{4.7}. Namely, there exist functions 
$I_1,I_2\in\cI(\dD)$ satisfying the equation
$U = \underline{I_2} \cdot \underline{I_1}^{-1}$.

\subsection*{Acknowledgment}

The authors are grateful to Professor V.E. Katsnelson, who carefully read an older version of this manuscript.
His comments and remarks very much influenced and improved the conception and presentation of the content 
of this paper.


\end{document}